\documentclass[12pt]{article}

\usepackage{amsmath, amsthm, amssymb}
\usepackage{color, soul}  
\newtheorem{lemma}{Lemma}
\newtheorem{defi}{Definition}
\newtheorem{thm}{Theorem}
\newtheorem{coro}{Corollary}
\newtheorem{prop}{Proposition}
\newtheorem{rem}{Remark}

\def\<{\langle}
\def\>{\rangle}
\def\Ha{\mathcal H}
\def\Ka{\mathcal K}

\def\Fe{\mathcal F}
\def\Me{\mathcal M}
\def\spn{{\rm span}\,  }
\def\Ne{\mathcal N}

\def\ha{\mathfrak h}
\def\Tr{\mathrm{Tr}\,}
\def\rhoinv{\rho^{\rm {inv}}}
\def\states{\mathfrak S}

\def\ketbra#1#2{|{#1}\rangle\langle{#2}|}

\title{On period, cycles and fixed points of a quantum channel}
\author{
Raffaella Carbone
\\Dipartimento di Matematica dell'Universit\'a di Pavia\\ via Ferrata 1, 27100 Pavia, Italy\\ \texttt{raffaella.carbone@unipv.it}
\\
Anna Jen$\check {\rm c}$ov\'a 
\\
Mathematical Institute of the Slovak Academy of Sciences 
\\
$\check {\rm S}$tef\'anikova 49, 814 73 Bratislava,
Slovakia\\
\texttt{jenca@mat.savba.sk}
}
\begin{document}
\maketitle

\begin{abstract}
We consider a quantum channel acting on an infinite dimensional von Neumann algebra of operators on a separable Hilbert space.
When there exists an invariant normal faithful state,
the cyclic properties of such channels are investigated passing through the decoherence free algebra and the fixed points domain. Both these spaces are proved to be images of a normal conditional expectation so that their consequent atomic structure are analyzed in order to give a better description of the action of the channel and, for instance, of its Kraus form and invariant densities.
\end{abstract}


\section{Introduction}

Quantum channels are basic tools in quantum theory. As a representation of 
a communication channel, they play a central role in quantum information theory and quantum information processing.
 They are seen as the counterpart of Markov operators in the non commutative models and they are generally used to represent the evolution of an open quantum system in discrete time models.

Some classical results related to Markov chains still need to be clarified in their non commutative version and the quantum theory reveals to be richer and more complicated due to the different framework and 
 techniques we are dealing with. In particular,  
 while the fixed points are a quite natural topic, already extensively studied also in the quantum case, 
the cyclic behavior (related to what is classically called period for a Markov chain) has still many mysterious aspects, starting from the fact that a good definition of a period for any irreducible quantum channel is not recognized by now; moreover these cycles have showed some typically non commutative features.
Both objects (cycles and fixed points), however, display a kind of rigidity in the structure of the channel which can link different irreducible components of the evolution; this is a strongly quantum feature in the sense that it is something that cannot be observed in a purely classical context.

Some aspects of this rigidity were already known and were object of interest in many papers in the last years,  related to different problems: e.g. the structure of the invariant states and irreducible decompositions \cite{BN,capa2016RMP}, decoherence free algebra and environmental decoherence \cite{CSU, DFSU}, the notion of sufficiency in quantum statistics \cite{JencovaPetz, petz1988, luczak}, periodicity and ergodic properties \cite{capa2015}.  In finite dimension, the structure of the channel and its spectrum, cycles and multiplicative properties were investigated in \cite{WolfLN,WPG}. In particular, multiplicative properties were studied  in view of applications to quantum information theory,  such as quantum error correction and private subspaces (e.g. \cite{CJK2009, JK2011, R}) or entanglement breaking channels \cite{RJP2018}.

In the present paper, a quantum channel is a  unital normal completely positive operator $\Phi$ on the algebra $B(\Ha)$ of the bounded operators on a separable (infinite dimensional) Hilbert space $\Ha$. In this setting, we study the subspace of fixed points and the so called decoherence free algebra (DFA) of the channel. The aim is to  obtain a unified description of these spaces and their relations, together with  the restricted action of the channel, in the presence of a faithful normal invariant  state.

Under the last assumption, the fixed point subspace is easily seen to be a subalgebra, moreover, it follows by the mean ergodic theorem for quantum dynamical systems \cite{frigerioverri,KN} that it is the range of a (faithful normal) conditional expectation, contained in the $w*$-closed convex hull 
 of the semigroup generated by the channel. Let us remark that the situation is more complicated in the general case, where the fixed point subspace may be not a subalgebra, see e.g. \cite{AGG,BJKW} for some descriptions, constructions and examples.

The second main object of interest, the decoherence free algebra of the channel, can be  informally defined as the largest subalgebra on which  $\Phi$ acts as a *-endomorphism. For dynamical semigroups of channels, the DFA  was  a very popular object already in the 70's and 80's, extensively used in order to study asymptotic properties of the semigroup (see e.g.\cite{frigerioverri, KN,robinson}), in particular, together with the fixed points, in order to distinguish between ergodicity and mean ergodicity.

More recently, the DFA appeared again in the literature  and was reconsidered because of the interest in reversible subsystems arising in quantum information and in relation with environmental decoherence, as defined in \cite{BO}. 
Most of these previous studies are generally concentrated in the case of a continuous time Markov semigroup. For instance, 
in \cite{DFR} a characterization of fixed points and of the DFA was found  in terms of the Lindblad form of the generator of the quantum dynamical semigroup. The DFA also appears in \cite{CSU} or \cite{Hellmich}, linked  with environmental decoherence and other forms of decompositions of the algebra.

Assuming the existence of a faithful invariant state, the analysis of the peripheral eigenvectors and 
a structural approach to the Perron-Frobenius spectral theory in  \cite{Groh84}, and more recently and in more generality in
\cite{BatkaiEtAl}, produce
the opportunity to split the algebra into a ``stable'' and ``reversible'' part with respect to the semigroup  (a Jacobs-DeLeeuw-Glicksberg type decomposition). The reversible part is a subalgebra 
 spanned by the peripheral  eigenvectors and it is the range of a (faithful normal) conditional expectation commuting with 
 the channel. This subalgebra is contained in the DFA and it is easily seen that in finite dimensions, these two subalgebras coincide,
  \cite{WolfLN}.

As one of our main results, we prove that, { for a channel acting on an atomic von Neumann algebra and with a faithful invariant state,  the reversible subalgebra coincides with DFA also in infinite dimensions.  Note that this implies that the DFA is the range of a conditional expectation, in particular, it is atomic. This allows us  to deduce the structural properties of the DFA and the action of the channel; in particular, we obtain a decomposition of the channel 
 into blocks with a finite cyclic structure. On the other hand, existence of a conditional expectation or more generally the atomicity of DFA  was commonly
assumed as a hypothesis (see \cite{BO,CSU,DFSU}) that allowed to obtain more precise results on environmental decoherence and the structure of the semigroup. Notice that our proof can be applied also to the continuous time case, so we can conclude that these results hold more generally. Furthermore, one could use the conditional expectation for the study of the decoherence time and spectral gap inequalities as in 
\cite{bardet2017,BR2018} in finite dimesions. These possiblities are remarked on but not pursued further in the paper,  and  left for future work.

The paper is structured as follows.
Section \ref{section:spaces} contains a characterization of the fixed points and of the DFA as commutants of suitable algebras defined in terms of the Kraus operators of the channel. This can be seen as a discrete time counterpart of \cite{DFR} for quantum dynamical semigroups, where the two spaces are characterized using the Lindblad form of the generator.
The relation between the DFA and the reversible subalgebra is also proved here  (Theorem \ref{thm:cjd}).

Afterward, in Section \ref{section:cycles},  we introduce the study of cycles, using also the notion of period, as introduced in the quantum context in \cite{FP}, and generalizing it to the infinite dimensional case. We start from the irreducible case, where the relation between the DFA, the fixed points of powers of the channel, and the cyclic decomposition is evident and can be clearly described (Corollary \ref{coro:W1} and Proposition \ref{prop:period_irred}). Then we turn to the reducible case, where  
we exploit the fact that the two algebras are atomic,  to deduce ad-hoc decompositions of the invariant states, relations with the Kraus operators, a better description of the conditional expectations and of the cyclic behavior of the channel (Proposition \ref{prop:dfs} and Theorem \ref{prop:fixed_multi}).
These results are strictly related to the studies in \cite{BN}, \cite{DFSU}  and the decomposition appearing in the last part of \cite{WPG}.
Finally, in Section \ref{section:OQRWs}, we apply our results to analyze a remarkable family of quantum channels, i.e. the so called open quantum random walks. This will give us the opportunity to show in details some explicit examples: in the last one, we try to throw a glance to a channel without invariant state. 


\section{Multiplicative domain, { decoherence free algebra } and fixed points} \label{section:spaces}

Let $\Ha$ be a separable Hilbert space. We will denote the algebra of bounded operators over $\Ha$ by $B(\Ha)$, the predual of $B(\Ha)$ by $B(\Ha)_*$  and  the set of normal states on $B(\Ha)$ by $\states(\Ha)$. The predual $B(\Ha)_*$ will be identified with the set of trace-class operators in $B(\Ha)$ and then $\states(\Ha)$ is the set of positive operators with unit trace. The identity operator on $\Ha$ will be sometimes denoted by $I_\Ha$ if the space $\Ha$ has to be emphasized. 

The main object studied in this paper is a unital normal completely positive map $\Phi:B(\Ha)\to B(\Ha)$, such maps are called 
(quantum)  channels. The preadjoint of $\Phi$ is the map $\Phi_*: B(\Ha)_*\to B(\Ha)_*$, defined by 
\[
\Tr [\Phi_*(\rho)A]=\Tr [\rho\Phi(A)],\qquad \rho\in B(\Ha)_*,\ A\in B(\Ha).
\]
The preadjoint of a channel is completely positive and preserves trace.

It is well known that any channel  $\Phi$ has a representation of the form
\begin{equation}\label{eq:kraus}
\Phi(A)=\sum_{k=1}^\infty V_k^*AV_k,\qquad A\in B(\Ha),
\end{equation}
where the Kraus operators $V_k\in B(\Ha)$ are such that $\sum_k V_k^*V_k=I$. Let $\Ka$ be a separable Hilbert space and let $\{e_k\}$ be an orthonormal basis of $\Ka$.
Let us define 
\begin{equation}\label{eq:isom}
V=\sum_k V_k\otimes |e_k\>,
\end{equation}
then $V: \Ha\to \Ha\otimes \Ka$ is an isometry and we obtain the Stinespring representation
\begin{equation}\label{eq:stinespring}
\Phi(A)=V^*(A\otimes I_\Ka)V,\qquad A\in B(\Ha).
\end{equation}

We will consider the following sets of operators:
\begin{itemize}
\item the fixed points' domain
$$\Fe(\Phi):=\{A\in B(\Ha), \Phi(A)=A\};
$$
\item the multiplicative domain
\begin{align}\label{def:multdomain}
\Me(\Phi)&:= \{ A\in B(\Ha), \Phi(A^*A)=\Phi(A)^*\Phi(A), \Phi(AA^*)=\Phi(A)\Phi(A)^*\}\nonumber
;
\end{align}
\item the decoherence free algebra (DFA)
$$\Ne(\Phi):=\cap_n \Me(\Phi^n).
$$
\end{itemize}

Since the map $\Phi$ will be fixed throughout, we will mostly use  the notations $\Me=\Me(\Phi)$, $\Ne=\Ne(\Phi)$ and $\Fe=\Fe(\Phi)$.

We now collect some basic facts about these sets. The proofs are included for the convenience of the reader.

First, notice that the set of fixed points is in general not a subalgebra (in contrast, as we will see, to $\Me$ and $\Ne$). An example can easily be constructed simply using a classical Markov chain with a transient class which can have access to two different positive recurrent classes. For quantum examples and discussion around the characterization of $\Fe$ and the following proposition, see e.g. \cite{AGG} or \cite[Section 3]{BJKW}.

\begin{prop}\label{prop:fixed} $\Fe$ is a von Neumann algebra if and only if it is contained in $\Ne$. 
In this case, we have
\[
\Fe=\{V_j,V_j^*, \ j=1,2\dots\}',
\]
where $\{\ \}'$ denotes the commutant.
\end{prop}

\begin{proof} The first statement is quite obvious. Assume now that $\Fe$ is a von Neumann algebra and let $A\in \Fe$. Then
\[
0=\Phi(A^*A)-A^*A= \sum_j(V_jA-AV_j)^*(V_jA-AV_j),
\]
this implies  $AV_j=V_jA$. Similarly, we obtain $AV_j^*=V_j^*A$. It follows that 
$\Fe\subseteq \{V_j,V_j^*, \ j=1,2\dots\}'$.
 The converse inclusion is clear. 
\end{proof}

We point out that when there is a faithful normal invariant state, then $\Fe$ is included in $\Ne$ and so the previous characterization holds.

We now turn to the multiplicative domain $\Me$. It was proved by Choi \cite{choi}  that $\Me$  satisfies the following multiplicative property 
\begin{equation}\label{MultProp}
\Phi(AB)=\Phi(A)\Phi(B),\ \Phi(BA)=\Phi(B)\Phi(A),\ \mbox{for all } A\in B(\Ha) \mbox{ and } B\in \Me.
\end{equation}
Consequently, $\Me$  is a von Neumann subalgebra in $B(\Ha)$ and the restriction of $\Phi$ to $\Me$ is a *-homomorphism. 
 We have the following characterization of $\Me$.

\begin{prop} \label{prop:multiplicative}  Let $V_1,V_2,\dots$ be Kraus operators as in \eqref{eq:kraus}.Then 
\[
\Me=\{ V_jV_k^*,\ j,k =1,2,\dots\}'.
\]

\end{prop}

\begin{proof} It will be convenient to use the Stinespring representation \eqref{eq:stinespring}. Let $V$ be as in \eqref{eq:isom} and  let $P=VV^*$, then $P\in B(\Ha\otimes \mathcal K)$ is a projection and we have $A\in\Me$ if and only if $A\otimes I$ commutes with $P$. Indeed, suppose $A\in \Me$, then 
\[
V^*(A^*A\otimes I)V=V^*(A^*\otimes I)P(A\otimes I)V.
\]
 It follows that $P(A^*\otimes I)(1-P)(A\otimes I)P=0$, hence $(1-P)(A\otimes I)P=0$, so that 
\[
(A\otimes I)P=P(A\otimes I)P.
\]
 Similarly, we get the same for $A^*$ and this implies that
\[
P(A\otimes I)=P(A\otimes I)P=(A\otimes I)P.
\]
The converse is easy. Now notice that  $P=\sum_{j,k}V_jV_k^*\otimes |e_j\>\<e_k|$, this implies the statement.

\end{proof}

It is clear from the definition that the DFA  $\Ne$ is a von Neumann subalgebra as well and it is also easy to see that $\Ne$ is the smallest subalgebra such that the restriction $\Phi|_\Ne$ is a *-endomorphism.

 \begin{rem}
 Notice that $\Phi|_{\Ne}$ is not
 always a *-automorphism. Indeed, $\Ne$ can have, for instance, a non-trivial intersection with the kernel of $\Phi$. 
 Since this intersection is a subalgebra, it then contains a nonzero projection $0\ne P\in \mbox{Ker} (\Phi)\cap \Ne$.
 On the other hand, any projection in $\mbox{Ker} (\Phi)$ is necessarily in $\Ne$, so that this happens if and only if  $\Phi$ is not faithful.

 \end{rem}

\begin{prop}\label{prop:dfageneral} We have the following characterizations of $\Ne$:
\begin{enumerate}
\item[(i)] $\Ne=\{V_{i_1}\dots V_{i_n}V_{j_1}^*\dots V_{j_n}^*,\ i_k,j_k=1,2,\dots; \ n\in \mathbb N\}'$.
\item[(ii)]  $\Ne$ is  the von Neumann algebra generated by the preserved projections, i.e. by the set
$$
 \{Q\in B(\Ha) \,:\,  \Phi^n(Q) \mbox{ is a projection }\forall \, n\ge 0 \}.
$$ 
\end{enumerate}
\end{prop}

\begin{proof}
(i) is immediate from Proposition \ref{prop:multiplicative}. (ii) holds since a projection  $Q$ is in  $\Me$ if and only if $\Phi(Q)$ is again a  projection.
\end{proof}

Point (ii) already appeared in \cite{CSU} (see also references therein) and was the original representation of the decoherence free algebra used in \cite{BO} when introducing environmental decoherence.

The following results are well known.

\begin{prop}\label{prop:faithful_invariant} Assume that there is a faithful normal invariant state for $\Phi$. Then
\begin{enumerate}
\item[(i)] $\Fe$ is a von Neumann subalgebra.
\item[(ii)] The restriction $\Phi|_{\Ne}$ is a *-automorphism.
\end{enumerate}
\end{prop}

\begin{proof} See e.g.  \cite{robinson} and references therein.  
\end{proof}

\subsection{Maps with a faithful invariant state}

In this section, we assume that  there is a faithful normal state $\rho\in \states(\Ha)$ for $\Phi$. 
In this case, there is another special subalgebra investigated in the literature, e.g. \cite{BatkaiEtAl, Groh84}, appearing in some asymptotic splitting, usually called the {\it reversible subalgebra} and denoted by ${\mathcal M}_r$. 
We describe the reversible subalgebra, following \cite{BatkaiEtAl}, \cite{Groh84} or \cite{Hellmich}.  Let $\mathbf S$ be the closure of the semigroup of channels  $\{\Phi^n,\ n\in \mathbb N\}$  in the point-ultraweak topology and define
$$
\Me_r:=\{X\in B(\Ha),\ T(X)^*T(X)=T(X^*X),\forall T\in {\mathbf S}\}.
$$

We will show below, in Theorem \ref{thm:cjd}, that the equality $\mathcal M_r= \Ne$ holds for channels on $B(\Ha)$ (or more generally on atomic von Neumann algebras). 

Due to the presence of a faithful normal invariant state, for any $\varphi\in B(\Ha)_*$, the set
 \[
\{\Phi_*^n\varphi, \ n\in \mathbb N\}
\]
is weakly relatively compact, equivalently, the set  $\mathbf S$ consists of normal operators  and is a compact semitopological 
semigroup (\cite[Proposition 2.1]{Groh84}). Further,  ${\mathbf S}$ contains a minimal ideal $M({\mathbf S})$ which is a compact topological group. Let $F$ be the unit of this group, then $M({\mathbf S})=F\circ {\mathbf S}$ and $F$ is a normal conditional expectation preserving the invariant state $\rho$ such that  $T F=F T$ for all $T\in {\mathbf S}$. Finally, $\Me_r$ is a von Neumann algebra 
and the minimal ideal $M({\mathbf S})$ acts as a compact group of *-automorphisms on $\Me_r$ (\cite[Theorem 1.2 and Corollary 1.3]{BatkaiEtAl}). 

This last fact trivially implies, in particular, that $\mathcal M_r\subseteq \Ne$ and that equality holds in finite dimension, but the infinite dimensional case is quite more delicate and tricky.

Now let $X\in B(\Ha)$ and let $O_0(X):=\{\Phi^k(X), k\in \mathbb N\}$ be the orbit of $X$ under $\{\Phi^k\}_{k\ge 0}$. Then 
the weak*-closure $\bar O_0(X)$ is the orbit of $X$ under $\mathbf S$, 
\[
\bar O_0(X)=O(X):=\{T(X), T\in {\mathbf S}\}
\]
and we can define the {\it stable subspace} as 
$$
\Me_s:=\{X\in B(\Ha),\ 0\in O(X)\}.
$$

The following lemma can be deduced from  \cite[Theorem 2.1]{Hellmich}, \cite[Theorem 1.2]{BatkaiEtAl} and 
  \cite[Proposition 2.2]{Groh84}.  Since we did not find an explicit and 
comprehensive statement in the literature, we reconstruct here the detailed result that we need and the proof for the convenience of the reader.

\begin{lemma}
Let $F$ be the normal conditional expectation introduced before.
\begin{enumerate}
\item $\Me_r=F(B(\Ha))=\overline{\mathrm{span}\{X\in B(\Ha),\ \Phi(X)=\lambda X,\ |\lambda|=1\}}^{w*}$;
\item $\Me_s=Ker (F)$.
\end{enumerate}
\end{lemma}

\begin{proof} 
1. Let us denote the last set on the RHS by $\Me_0$.
We will prove the chain of inclusions 
\[
\Me_r\subseteq F(B(\Ha))\subseteq \Me_0\subseteq \Me_r.
\]

First, if $X\in \Me_r$, then since $F\in {\mathbf S}$, we have $F(X^*X)=F(X)^*F(X)$. This implies 
\[
F((X-F(X))^*(X-F(X)))=0
\]
and since $F$ is faithful, we have $X=F(X)\in F(B(\Ha))$.

To prove the second inclusion, let $X=F(X)$. Let $\widehat{M({\mathbf S})}$ be the dual group and let $\chi\in \widehat{M({\mathbf S})}$. Let us define
\[
X_\chi:= \int_{M({\mathbf S})}\overline{\chi(T)}T(X)d\mu(T),
\]
where $\mu$ is the normalized Haar measure over $M({\mathbf S})$. The integral is defined in the weak*-topology, 
so we have $X_\chi\in B(\Ha)$ and  since for $T\in M({\mathbf S})$, we have $T=TF=FT$, we obtain 
\begin{align*}
\Phi(X_\chi)&=\int_{M({\mathbf S})}\overline{\chi(T)}\Phi  T(X)d\mu(T)=\int_{M({\mathbf S})}\overline{\chi(T)}F \Phi  T(X)d\mu(T)\\
&=
\chi(F\Phi )\int_{M({\mathbf S})}\overline{\chi( T)} T(X)d\mu(T)=\chi(F\Phi) X_\chi
\end{align*}
(since $F\Phi=\Phi F\in M({\mathbf S})$), so that $X_\chi\in \Me_0$.
Let now $\psi\in B(\Ha)_*$ be such that $\Tr [\psi Y]=0$ for any peripheral eigenvector $Y$ of $\Phi$. Then
\[
0=\Tr [\psi X_\chi]=\int_{M({\mathbf S})}\overline{\chi(T)}\Tr [\psi T(X)]d\mu(T), \qquad \forall \chi\in \widehat{M({\mathbf S})}.
\]
 Since the characters span the space of square integrable functions on $M({\mathbf S})$ and the function 
 $T\mapsto \Tr [\psi T(X)]$ is continuous, it follows that $\Tr [\psi T(X)]=0$ for all $T\in M({\mathbf S})$, in particular, 
 $\Tr [\psi X]=\Tr [\psi F(X)]=0$.
This implies that $\{Y,\ \Phi(Y)=\lambda Y,\ |\lambda|=1\}^\perp\subseteq F(B(\Ha))^\perp$, so that
\[
F(B(\Ha))=F(B(\Ha))^{\perp\perp}\subseteq \{Y,\ \Phi(Y)=\lambda Y,\ |\lambda|=1\}^{\perp\perp}=\Me_0.
\]

Finally, let $\Phi(X)=\lambda X$ for some $|\lambda|=1$, then $\Phi^k(X)=\lambda^k X$
 for all $k\in \mathbb N$. Let $S\in {\mathbf S}$, then there is a net $\Phi^{n_\alpha}\to S$,
 so that $S(X)=\lim \Phi^{n_\alpha}(X)=\lim \lambda^{n_\alpha} X$, hence  $S(X)=\mu X$  with $\mu=\lim \lambda^{n_\alpha}$.  By Schwartz inequality, $S(X^*X)\ge S(X)^*S(X)=X^*X$ and applying the faithful normal invariant state $\rho$ we obtain $S(X^*X)=X^*X$, so that $X\in \Me_r$. This proves the last of the above chain of inclusions.

2. Since $F\in M({\mathbf S})$, we clearly have ${\rm Ker} F\subseteq \Me_s$. Conversely, let $X\in  \Me_s$ and let 
$S\in {\mathbf S}$ be such that $S(X)=0$. Since $FS\in M({\mathbf S})$, there is some $T\in M({\mathbf S})$ such that $TFS=F$, so that we have
\[
F(X)=TFS(X)=0.
\]
This concludes the proof.
\end{proof}

 

We will now prove the main result of this section.

\begin{thm}\label{thm:cjd} Assume that a quantum channel $\Phi:B(\Ha)\to B(\Ha)$ admits a faithful normal invariant state $\rho$. Then 
$\Ne=\Me_r$.

\end{thm}

\begin{proof} Let $\mathcal B_1$, $\Ne_1$ and $\mathcal R_1$ be the unit balls of $B(\Ha)$, $\Ne$ and $\Me_r$, respectively.
 Then 
\[
{\mathcal R}_1\subseteq \Ne_1\subseteq \bigcap_n \Phi^n(\mathcal B_1).
\]
Indeed, the first inclusion follows from $\Me_r\subseteq \Ne$ and the second from the fact that the restriction 
$\Phi|_{\Ne}$ is an automorphism. 
We will show that ${\mathcal R}_1=\bigcap_n \Phi^n(\mathcal B_1)$, which implies the statement. (This proof is inspired by \cite{Arveson}.)

We will use a Hahn-Banach separation argument. So let $X\in \bigcap_n \Phi^n(\mathcal B_1)\setminus {\mathcal R}_1$. Since ${\mathcal R}_1\subset B(\Ha)$ is convex and compact in the weak*-topology, there exists some $\psi\in B(\Ha)_*$ such that 
 \[
\Tr [\psi X]>\sup_{Y\in {\mathcal R}_1} \Tr [\psi Y]=\|\psi|_{\Me_r}\|_1=\|F_*\psi\|_1.
 \]
For each $n\in \mathbb N$, there is some $Y_n\in \mathcal B_1$ such that $X=\Phi^n(Y_n)$ and we have
\[
\|\Phi^n_*\psi\|_1\ge \Tr [\Phi^n_*(\psi)Y_n]=\Tr [\psi X].
\]
Note that since $\Phi_*$ is a contraction, $\{\|\Phi^n_*\psi\|_1\}_n$ is a bounded nonincreasing sequence and we have
\[
\lim_n \|\Phi^n_*\psi\|_1\ge \Tr [\psi X]>\|F_*\psi\|_1.
\]

On the other hand,  for any $\varphi\in B(\Ha)_*$, the orbit
\[
\mathbf S_*\varphi: = \{S_*\varphi,\ S\in \mathbf S\}=\{\Phi_*^n\varphi, \ n\in \mathbb N\}^{-w}
\]
is weakly  compact. Since $F\in \mathbf S$, $\mathbf S_*\varphi$ contains $F_*\varphi$ and since $B(\Ha)_*$ is a separable Banach 
space, the weak topology on the orbit is a metric topology (\cite[Theorem V.6.3]{dunford-schwartz}). Hence
there is a subsequence of $\Phi^{n}_*\varphi$ converging weakly to $F_*\varphi$. 

Let $\Phi^{n_k}$ be such that $\Phi^{n_k}_*\psi\to F_*\psi$ and let $\varphi_1,\dots,\varphi_4\in B(\Ha)_*^+$ be such that $\psi=\sum_i c_i\varphi_i$. Then  we may assume that $\Phi^{n_k}_*\varphi_i$ are all weakly convergent, restricting to subsequences if necessary
(\cite[Theorem V.6.1]{dunford-schwartz}).   By \cite[Corollary III.5.11]{takesakiI},  $\Phi^{n_k}_*\varphi_i$ are all norm convergent. It follows that $\Phi^{n_k}_*\psi\to F_*\psi$ in norm, so that 
\[
\lim \|\Phi^n_*\psi\|_1=\lim \|\Phi^{n_k}_*\psi\|_1=\|F_*\psi\|_1,
\]
a contradiction.

\end{proof}

\begin{coro}\label{coro:range} Let $\Phi: B(\Ha)\to B(\Ha)$ be a channel admitting a faithful normal invariant state  $\rho$. 
Then $\Ne$ is the range of a normal conditional expectation $E_\Ne$ preserving $\rho$ and commuting with $\Phi$. Consequently, $\Ne$ is an atomic von Neumann algebra.
\end{coro}

\begin{proof} Put $E_\Ne= F$ and the fact that $\Ne$ must be  atomic follows by \cite{tomiyama}.
\end{proof}

\begin{rem}  Note  that Theorem \ref{thm:cjd} and 
Corollary \ref{coro:range} hold for quantum channels on any atomic von Neumann algebra $\mathcal M$. The same proof can be used also in continuous time case.
\end{rem}

\begin{rem}\label{rem:EID}
As we mentioned in the introduction, the DFA is a basic object in the study of the problem of environmental decoherence.
According to the theory introduced by Blanchard and Olkiewicz (\cite{BO}), a system undergoing the evolution 
$(\Phi^n)_n$ displays environmental decoherence if there exist two subspaces ${\mathcal M}_1$ and ${\mathcal M}_2$, both preserved by the channel, and such that
\begin{itemize}
\item $B(\Ha)={\mathcal M}_1\oplus {\mathcal M}_2$,
\item ${\mathcal M}_1$ is a von Neumann algebra and $\Phi$ is a *-automorphism when restricted to ${\mathcal M}_1$,
\item $w*-\lim_n \Phi^n(x)=0$ for all $x$ in ${\mathcal M}_2$.
\end{itemize}
The fact that ${\mathcal M}_1$ coincides with $\Ne$ and can be the image of a normal conditional expectation is in general an interesting but not clear point as far as we know (\cite{CSU}). The previous theorem allows us to prove that this is true whenever the channel has an invariant faithful density; moreover, in the same case, we can deduce there is environmental decoherence choosing the decomposition ${\mathcal M}_1=\Ne$ and ${\mathcal M}_2={\mathcal M}_s$.
This last consideration is an almost direct consequence stated for instance in \cite[Proposition 31]{CSU}.

Due to the previous remark, these conclusions hold also for the continuous time case, so for instance, it can generalize many of the results concerning EID for quantum dynamical semigroups as treated in \cite{CSU} (see in particular Section IV).

Finally, we emphasize that the existence of a conditional expectation with range $\Ne$, commuting with the channel, can be a useful tool to study the velocity of decoherence; but we shall come back to this point later, in Remark \ref{rem:SG}.
\end{rem}

\section{Cyclic decompositions}\label{section:cycles}

In this section, we shall investigate the cyclic behavior of a quantum channel. We shall start with irreducible maps: here the cycles can be analyzed using the period and we can prove that the DFA is commutative. Then we shall go to the general case, where the study of cycles is more demanding.

Following conventional terminology already introduced in the 70s  (see \cite{EHK} and references therein), we say that the map  $\Phi$ is irreducible if there are 
no proper hereditary subalgebras preserved by the channel; equivalently, if there exist 
no nontrivial subharmonic projections, that is, if $P\in B(\Ha)$ is a projection such that $\Phi(P) \ge P$ then 
 $P=0$ or $P=I$. If there is a faithful normal invariant state, this clearly happens if and only if $\Fe=\mathbb CI$. 
Moreover, it follows by the Perron-Frobenius theory for positive map on trace class operators {\cite{Schrader} that there is at most one invariant faithful state for irreducible $\Phi$.
 
 \subsection{Irreducible quantum channels}\label{Subsection:IQC}
 
 We shall concentrate here on irreducible quantum channels with an invariant faithful state.
 In this case, the cycles of the channel are clearly related to the decoherence free algebra, we can use the notion of period (which consists in a precise structure of the peripheral spectrum of the channel), and this will give a precise link with the fixed points domain of the powers of the channel.

First, we introduce the definition of period as was made for the finite dimensional case in \cite{FP} (but see also \cite{EHK} and \cite{Schrader}).

\begin{defi}\label{def:period} Period of $\Phi$. 
Let $\Phi$ be an irreducible quantum channel. Then the period $d$ is the maximal integer $d$ such that there exists a resolution of the identity 
$Q_0,\dots,Q_{d-1}$ verifying $\Phi(Q_j)=Q_{j-1}$ for all $j$ (subtractions on indices are modulo $d$).
\\
Each $Q_j$ is called a cyclic projection and the set $Q_0,...Q_{d-1}$ will be called cyclic decomposition (or cyclic resolution) of $\Phi$.
\end{defi}

This is a good definition in the context of finite dimensional Hilbert spaces. When we work on infinite dimensional spaces, we need to prove that (or when) the period $d$ is finite. For this, we need to use some spectral properties of the channel.

\begin{prop} [Groh \cite{Groh84} and Batkai et al {\cite[Propositions 6.1 and 6.2]{BatkaiEtAl}}] \label{prop:period} Let  $\Phi$ be an irreducible quantum channel on $B(\Ha)$ with an invariant faithful state.
Then the peripheral point spectrum of $\Phi$ is the group of all the $d$-th roots of unity for some $d\ge 1$ 
and all the eigenvalues in the peripheral point spectrum are simple.
Moreover there exists a unitary operator $U$ such that $U^d=I$ and $\Phi(U^n)=\exp(i2\pi n/d)U^n$ for all integer $n$.
\end{prop}

In the finite dimensional case, this result was proved in \cite{EHK}. Here the existence of a faithful invariant state is not necessary and 
 it is enough to assume that $\Phi$ is a Schwarz map. On the other hand, \cite[Example 1.3]{Groh82} shows that if the map is only positive, the
 peripheral spectrum may  not be a subgroup of the unit circle.

\begin{coro}\label{coro:W1}
Let  $\Phi$ be an irreducible quantum channel on $B(\Ha)$ with an invariant faithful state.
Then $\Phi$ has finite period, the cyclic resolution of $\Phi$ is unique and $\Ne$ is an abelian algebra spanned by the cyclic projections of $\Phi$.
\end{coro}

\begin{proof}
Let $\omega$ be the primary $d$-th root of unity and $U$ the  unitary operator satisfying $U^d=I$ and $\Phi(U^n)=\omega^n U^n$ of Proposition \ref{prop:period}. It follows that $U^n$ is the unique (up to multiplicative constants) eigenvector associated with the eigenvalue $\omega^n$.
 By Theorem \ref{thm:cjd},
\[
\Ne=\Me_r=\spn\{I, U,\dots U^{d-1}\}=\{U\}''.
\] 
In particular, it follows that the abelian subalgebra generated by $U$ is finite dimensional and   $U$ admits a spectral representation
$$
U= \sum_{j=0}^{d-1} \omega^j Q_j
$$
for some orthogonal projections $Q_j$ summing up to $I$. We immediately deduce that, since $\Phi(U)=\omega U$, then $\Phi(Q_j)=Q_{j-1}$ for all $j$, so that $Q_0,..Q_{d-1}$ is a cyclic decomposition of $\Phi$ and we have
\[
\Ne=\{U\}''=\spn\{Q_0,\dots,Q_{d-1}\}.
\] 
To prove uniqueness, assume that $P_0,\dots, P_{d-1}$ is another cyclic resolution of $\Phi$. Then we can construct the unitary operator
$$
V = \sum_{j=0}^{d-1} \omega^j P_j,
$$
which is an eigenvector for $\Phi$ corresponding to $\omega$. Since the eigenvalues are simple, we must have $V=zU$ for some $z\in \mathbb C$, $|z|=1$ and it is easy to see that $z=\omega^k$ 
for some $k\in \{0,\dots,d-1\}$ and then for each $j$,  
we must have $P_j=Q_{j-k}$ (subtraction modulo $d$).

\end{proof}

\begin{prop}\label{prop:period_irred}
Suppose $\Phi$ is an irreducible quantum channel with an invariant faithful state and let $Q_0,\dots,Q_{d-1}$ be the cyclic resolution for $\Phi$. 
Then
\begin{enumerate}
\item ${\cal F}(\Phi^m)$ is a subalgebra of $\Ne$ for any $m$; 
\item ${\cal F}(\Phi^d)=\Ne$ and $d$ is the smallest integer with this property;
\item $\Fe(\Phi^m)= \mathbb C1$ if and only if $GCD(m,d)=1$.
\end{enumerate}
Moreover, denote by $\Phi^d_{|k}$ the restriction of $\Phi^d$ to the subalgebra $Q_k B(\Ha) Q_k$, then $\Phi^d_{|k}$ is irreducible, positive recurrent and aperiodic, and consequently ergodic.

\end{prop}

\begin{proof}  Let $\rho\in \mathfrak S(\Ha)$ be the (unique) faithful invariant state of $\Phi$, then $\rho$ is also invariant for  $\Phi^m$, so that by Propositions 
\ref{prop:fixed} and \ref{prop:faithful_invariant}, $\Fe(\Phi^m)$ is a subalgebra in $\Ne(\Phi^m)$. Note that for any $n\in \mathbb N$ and  $X\in \Me(\Phi^n)$, we have by Schwarz inequality  that
\begin{eqnarray*}
\Phi^n(X^*X)&=&\Phi(\Phi^{n-1}(X^*X))\ge \Phi(\Phi^{n-1}(X)^*\Phi^{n-1}(X))\\
&\ge& \Phi^n(X)^*\Phi^n(X)=\Phi^n(X^*X).
\end{eqnarray*}
Using the fact that $\Phi^{n-1}(X^*X)-\Phi^{n-1}(X)^*\Phi^{n-1}(X)\ge 0$ and that $\rho$ is a faithful invariant state, we obtain that $\Phi^{n-1}(X^*X)=\Phi^{n-1}(X)^*\Phi^{n-1}(X)$. This implies that $\Me(\Phi^n)\subseteq \Me(\Phi^{n-1})$ for all $n$ and hence
\[
\Ne(\Phi^m)=\cap_{n} \Me(\Phi^{mn})=\cap_n\Me(\Phi^n)=\Ne.
\]
This proves 1. 

By definition of cyclic decomposition, we have $Q_j\in \Fe(\Phi^d)$ for all $j$, this implies $\Ne\subseteq \Fe(\Phi^d)$. The converse inclusion holds by part 1.
If $n<d$, then $\Phi^n(Q_{d-1})=Q_{d-n-1}\ne Q_{d-1}$, so that $Q_{d-1}\notin \Fe(\Phi^n)$ and hence $\Fe(\Phi^n)\ne \Ne$, this proves 2.

Assume now that $\Fe(\Phi^m)\ne \mathbb C1$, then 
 there is some nontrivial minimal projection $P\in \Fe(\Phi^m)$, 
which by part 1. must be of the form
$P=Q_{j_1}+\dots + Q_{j_k}$ for some (distinct) indices $0\le j_i\le d-1$ and $k<d$. 
Let $P_i=\Phi^i(P)$, $i=0,\dots,m-1$, then all  $P_i$ are minimal projections in $\Fe(\Phi^m)$, so that for $i\ne j$, either $P_iP_j=0$ or $P_i=P_j$. By rearranging the indices if necessary, we may assume that  $P_0,\dots,P_{l-1}$ are mutually orthogonal and all other 
$P_i$ are contained in $\{P_0,\dots,P_{l-1}\}$. Then $\sum_{i=0}^{m-1} P_i=\sum_{j=0}^{l-1} n_jP_j$ for some integers $n_j$.
 On the other hand, we have
$\sum_{i=0}^{m-1} P_i\in \Fe=\mathbb C1$ since $\Phi$ is irreducible. It follows that $n_1=\dots=n_{l-1}=:n$ and  
\[
\sum_{i=0}^{m-1} P_i=nI=n\sum_{j=0}^{l-1}P_j.
\]
This implies $m=nl$. Further, $\sum_{j=0}^{l-1}P_j =I$ implies that $d=kl$ by the definition of $P_j$.  Note also that $l>1$ since otherwise we would have $\Phi(P)=P$, 
which is not possible. Conversely, assume that $GCD(m,d)=l>1$ and let $d=kl$. Put $P=Q_0+Q_l+\dots+Q_{(k-1)l}$, then clearly $P$ is a projection, $P\ne 0,1$ and $\Phi^l(P)=P$ and also $\Phi^m(P)=P$, since $m$ is a multiple of $l$, so that $P\in \Fe(\Phi^m)$ and $\Fe(\Phi^m)\ne \mathbb C1$.  

To prove the last statement, observe that  $\Phi^d_{|k}$ is positive recurrent because the restriction of the $\Phi$-invariant state will give a faithful $\Phi^d_{|k}$-invariant state. By contradiction, if $\Phi^d_{|k}$ is reducible, then we have a non trivial $\Phi^d_{|k}$-harmonic projection $Q$, $0<Q < Q_k$, i.e. such that $\Phi^d(Q)=Q$.
But then this $Q$ is in $\Ne$ and, by positivity $\Phi^j(Q)$ is a projection bounded above by $\Phi^j(Q_k)=Q_{k-j}$. We deduce
that $\sum_{j=0}^{d-1}\Phi^j(Q)<\sum_{j=0}^{d-1}Q_{k-j}=1$ is a non trivial projection and a fixed point for $\Phi$ and this contradicts the irreducibility of $\Phi$.

Similarly, for the period, we know that $\Phi^d_{|k}$ has finite period by Corollary \ref{coro:W1}; we call its period $d_k$, 
with cyclic decomposition $R_0,... R_{d_k-1}$, $R_0+\cdots +R_{d_k-1}=Q_k$.
$R_0$ is a fixed point for $\Phi^{dd_k}$, so it belongs to $\Ne$ and $\Phi^j(R_0)$, $j=0,... d\cdot d_k-1$, will give a cyclic decomposition for $\Phi$. So $dd_k=d$ which implies $d_k=1$ and $R_0=Q_k$.
\end{proof}

\begin{rem}\label{rem:SG}
On the line of Remark \ref{rem:EID}, we can now give some more details on how Theorem \ref{thm:cjd} can help in evaluating the ``time for decoherence'' for an irreducible channel with an invariant faithful density. In particular, it is in general interesting to understand when the evolution of the channel tends to become reversible exponentially fast, or equivalently 
when the elements of the stable space ${\mathcal M}_s$ converge to $0$ exponentially fast and with uniform rate;
 this can be characterized using a kind of spectral gap parameter.

In the standard literature on this topic, the convergence is considered with respect to the $L^2$ structure induced by the invariant faithful density, say $\rho$, i.e. one usually defines a scalar product $\<x,y\>_\rho={\rm Tr}(\rho x^*y)$, and consequently a norm $\|x\|^2_{2,\rho}={\rm Tr}(\rho x^*x)$, for $x$ and $y$ in $B(\Ha)$. 
Then the suitable $L^2(\Ha,\rho)$ space will be the completion of $B(\Ha)$ with respect to this norm.
\begin{itemize}
\item The first fact worth to be noticed is that the good behavior of the conditional expectation $F$ given by Theorem \ref{thm:cjd} imply that it gives an orthogonal decomposition in $L^2(\Ha,\rho)$ and this orthogonality is preserved by $\Phi$, in the sense that
$$
\<Fx,(1-F)y\>=\<\Phi(Fx),\Phi((1-F)y)\>=0
\qquad \forall \, x,y.
$$

Indeed, for any $n$,
\begin{eqnarray*}
\<\Phi^n(Fx),\Phi^n((1-F)y)\>
&=& {\rm Tr} (\rho \Phi^n(Fx^*)\Phi^n((1-F)y))
\\
&=& {\rm Tr} (\rho \Phi^n((Fx^*)((1-F)y))
\\
&=& {\rm Tr} (\rho ((Fx^*)((1-F)y))
\\
&=& \<Fx,(1-F)y\>
\end{eqnarray*}
The previous, for $n=1$, gives the first equality and we can deduce the second taking $n=md+k$, where $d$ is the period of the channel and $k=0,...d-1$, since we can also write
\begin{eqnarray*}
\<\Phi^n(Fx),\Phi^n((1-F)y)\>
&=& {\rm Tr} (\rho \Phi^k(Fx^*)\Phi^{md+k}((1-F)y))
\rightarrow_{m\rightarrow \infty} 0
\end{eqnarray*}
and then repeat for all possible $k$.

\item Moreover $\Phi$ is contractive also with respect to this new norm due to the Schwarz property
$$
\| \Phi(x)\|_{2,\rho}^2
= {\rm Tr} (\rho \Phi(x^*)\Phi(x))
\le  {\rm Tr} (\rho \Phi(x^*x))
= {\rm Tr} (\rho x^*x)) =\| x\|_{2,\rho}^2;
$$
and it is isometric on the range $\Ne$ of $F$ since, for $x$ in $\Ne$, the multiplicativity property will transform the inequality in the previous line into an equality.
\end{itemize}

In this context, we could define the decoherence spectral gap as the maximum $\epsilon$ such that
$$
\|\Phi^n(x)-\Phi^n(Fx)\|_{2,\rho}\le e^{-\epsilon n}\|x-Fx\|_{2,\rho}
\qquad \mbox{for all $x$ and $n$}.
$$
The existence of a strictly positive $\epsilon$, uniform in $n$ and $x$, would give the uniform exponential convergence of the evolution to the decoherence-free algebra. Obviously, such optimal $\epsilon$ can also be characterized as 
$$
{\epsilon} = \inf_{n> 0,x\in{B(\Ha)}}\frac{1}{n}\ln \frac{\|\Phi^n(x)-\Phi^n(Fx)\|_{2,\rho}}{\|x-Fx\|_{2,\rho}}
= \inf_{n> 0,x \perp \Ne}\frac{1}{n}\ln \frac{\|\Phi^n(x)\|_{2,\rho}}{\|x\|_{2,\rho}}
$$

A common technique for estimating the usual spectral gap for finite classical Markov chains consists in using continuous time generators. The same ideas can be applied to our case of interest, with the proper adaptation.
Briefly:
\\- first, we consider the operator $\Phi^d$, where $d$ is the period of the channel, so that the algebra $\Ne$ is the fixed space of the new operator,
\\- second, we introduce the infinitesimal Lindblad generator ${\mathcal L}:=(\Phi^d-1)$, inheriting the invariant faithful state of $\Phi$, and we compute the spectral gap of ${\mathcal L}$ with the usual standard techniques.

Some interesting similar problems in the continuous setting have been studied in \cite{bardet2017} and \cite{BR2018}.
\end{rem}

\subsection{Reducible maps}\label{sec:reducible}

By Corollary \ref{coro:range}, if  $\Phi$  admits a faithful normal invariant state $\rho$,
the decoherence free algebra $\Ne$ is the range of a faithful normal  conditional expectation $E_\Ne$ and consequently must be atomic.
 On the other hand,  it is known \cite{frigerioverri, KN} that the limit
\begin{equation}\label{eq:limit}
\lim_n \frac1n \sum_{k=0}^{n-1}\Phi^k
\end{equation}
exists in the point-ultraweak topology and gives a faithful normal conditional expectation $E_\Fe$ onto $\Fe$, 
satisfying 
\begin{equation}\label{eq:condexp_min}
E_\Fe\circ\Phi=\Phi\circ E_\Fe=E_\Fe.
\end{equation}
Hence $\Fe$ is an atomic von Neumann subalgebra of $\Ne$. In this section, we study  the structure of the
channel induced by the two algebras $\Fe$ and $\Ne$.

First of all, we explain, in Lemma \ref{lemma:condexp}, how the minimal central projections of either $\Fe$ or $\Ne$ are related to a better description of the corresponding algebra, the action of the associated conditional expectation and its
invariant states. Then, in Proposition \ref{prop:dfs} and Theorem \ref{prop:fixed_multi}, we detail the study of the channel with respect to the structural properties of $\Fe$ and $\Ne$. This will lead to a simplified characterization of the channel, its  Kraus operators and  invariant states. The simplification essentially follows from the fact that $\Phi$ can be described by a collection of ``lower dimensional'' operators.

We first describe a general form of a faithful normal conditional expectation on $B(\Ha)$.

\begin{lemma}\label{lemma:condexp} Let $E:B(\Ha)\to B(\Ha)$ be a faithful normal conditional expectation and let $\mathcal R=E(B(\Ha))$ be its range. Then
\begin{enumerate}
\item[(i)] $\mathcal R$ is atomic, so that there is a direct sum decomposition $\Ha=\bigoplus_j \Ha_j$, Hilbert spaces $\Ha^L_j$, $\Ha^R_j$ and unitaries $U_j:\Ha_j\to \Ha_j^L\otimes \Ha^R_j$ such that 
\[
\mathcal R= \bigoplus_j U_j^* (B(\Ha_j^L)\otimes I_{\Ha_j^R})U_j;
\]
\item[(ii)]  the orthogonal projections $P_j$  onto $\Ha_j$ are minimal central projections in $\mathcal R$ and 
\[
E(A)=\sum_j E(P_jAP_j);
\]
\item[(iii)] identifying $P_jB(\Ha)P_j$ with $B(\Ha_j)$, the restriction of $E$ to $P_jB(\Ha)P_j$ is determined by
\[
E(U_j^*(A_j\otimes B_j)U_j)=U_j^*(A_j\otimes \Tr [\rho_jB_j] I_{\Ha_j^R})U_j,
\]
where each $\rho_j\in \states (\Ha^R_j)$ is a (fixed) faithful normal state 
\item[(iv)] a normal state $\omega\in \states(\Ha)$ is invariant under $E$ if and only if 
\[
\omega=\oplus_j\lambda_j U_j^*(\omega^L_j\otimes \rho_j)U_j,
\]
where $\rho_j$ are as in (iii),  $\lambda_j\in [0,1]$, $\sum_j\lambda_j=1$  and  $\omega^L_j\in \mathfrak S(\Ha_j^L)$.
\end{enumerate}

\end{lemma}

\begin{proof}
 The range $\mathcal R$ is atomic by \cite{tomiyama}. Let $\{P_j\}$ be the minimal central projections in $\mathcal R$ and let $\Ha_j=P_j\Ha$. Since $\mathcal R P_j$ is a type I factor acting on $\Ha_j$, there are Hilbert spaces $\Ha_j^L$, $\Ha_j^R$ and a unitary $U_j:\Ha_j\to \Ha_j^L\otimes \Ha_j^R$ such that 
\[
\mathcal R P_j=U_j^*(B(\Ha_j^L)\otimes I_{\Ha_j^R})U_j,
\]
 this proves (i). By the properties of conditional expectations, 
\[
E(P_jAP_k)=P_jE(A)P_k=P_jP_kE(A)
\]
 for any $A\in B(\Ha)$, this proves  (ii). It also follows that under the identification in (iii), $E(B(\Ha_j))\subseteq B(\Ha_j)$ for all $j$ and the restriction  of $E$ is a faithful normal conditional expectation on $B(\Ha_j)$, with range $U_j^*(B(\Ha_j^L)\otimes I_{\Ha_j^R})U_j$. Let  $B_j\in B(\Ha_j^R)$, then the multiplicative property of $E$ implies that 
 $E(U_j^*(I\otimes B_j)U_j)$ must  commute with all elements in $U_j^*(B(\Ha_j^L)\otimes I))U_j$. It follows  that 
there is some $c_j(B_j)\in \mathbb C$ such that $E(U_j^*(I\otimes B_j)U_j)=c_j(B_j)P_j$. It is clear that $B_j\mapsto c_j(B_j)$ defines a normal state on $B(\Ha_j^R)$ with corresponding density $\rho_j\in \states(\Ha^R_j)$, which must be faithful since $E$ is. This proves (iii).

For (iv), let $\omega\in \states(\Ha)$. It is clear that if $E_*(\omega)=\omega$, then we must have $\omega=\sum_j \lambda_j\omega_j$ for some $\omega_j\in \states (\Ha_j)$
 and $\lambda_j=\Tr [P_j\omega]$. Let $\omega^L_j\in \states(\Ha_j^L)$ be the partial trace  $\Tr_{\Ha_j^R}[U_j\omega_j U_j^*]$. Then $\omega_j$, and consequently also $\omega$, is invariant under $E$ if and only if for all $A_j\in B(\Ha^L_j)$ and $B_j\in B(\Ha^R_j)$, 
\[
\omega_j= E^*(\omega_j)=U_j^*(\omega^L_j\otimes \rho_j)U_j
\]
and this concludes the proof.
\end{proof}

The previous lemma, applied with $\mathcal R$ equal to either $\Fe$ or $\Ne$, will give us two different decompositions of the Hilbert space $\Ha$, into ranges of minimal central projections. We can better detail these two decompositions separately, exploiting their peculiar features, but we mainly want to fit the two together, in order to optimize our knowledge.  
Therefore, searching for the finest decomposition of $\Ha$ which contains both the previous decompositions takes us to consider the algebra
$$
\mathcal Z=\mathcal Z(\Fe)\cap \mathcal Z(\Ne),
$$
where $\mathcal Z(\Fe)$ and $\mathcal Z(\Ne)$ denote the centers of $\Fe$ and $\Ne$ respectively.
Clearly, $\mathcal Z$ is a discrete abelian von Neumann algebra and the minimal projections in $\mathcal Z$, say $\{Z_1,Z_2,\dots\}$, will be denumerable and give a resolution of the identity. We shall call any $Z_i$ a MFNC (minimal and $\Fe$/$\Ne$-central) projection.
Identifying $Z_iB(\Ha)Z_i$ with $B(Z_i\Ha)$, we have
$\Phi(B(Z_i \Ha))\subseteq  B(Z_i\Ha)$, so that
$\Phi_i:=\Phi|_{B(Z_i\Ha)}$ is a quantum channel on $B(Z_i\Ha)$, with  $\Ne(\Phi_i)=\Ne Z_i=:\Ne_i$; 
$\Phi_i$ will be denominated a MFNC component of the channel.

\begin{defi}\label{def:period2}
Period of a MFNC component.
\\
Let $\Phi$ be a MFCN component (or equivalently a quantum channel for which the algebra $\mathcal Z=\mathcal Z(\Fe)\cap \mathcal Z(\Ne)$ is trivial) 
with an invariant faithful density. Then the period of $\Phi$ is the dimension $d$ of the algebra $\mathcal Z(\Ne)$.
\\
Further, the collection of minimal projections $\{Q_m\}_{m\in {\mathbb Z}_d}\in {\mathcal Z}(\Ne)$ summing up to the identity and such that $\Phi(Q_m)=Q_{m- 1}$ is called cyclic resolution (or decomposition) for $\Phi$.
\end{defi}

Similarly as for the irreducible case, the first point will be to show that the period and the cyclic resolution of a MFNC component are well defined and unique. This will immediately be proven in Proposition \ref{prop:central}.

\begin{rem} 

1. If $\Phi$ is irreducible, $\mathcal Z$ is trivial, so that $\Phi$ itself is the unique MFNC component. Then Definitions \ref{def:period} and \ref{def:period2} coincide and will give the same period and cyclic resolution since ${\mathcal Z}(\Ne)=\Ne$. 

2. For reducible MFNC components $\Phi$, we will see later in Theorem \ref{prop:fixed_multi} that 
 $d$ is the period of all irreducible restrictions of the component $\Phi$.
\end{rem}

\begin{prop}\label{prop:central} 
Let $\Phi$ be a quantum channel with an invariant faithful density and let $\{Z_i\}_i$ and $\{\Phi_i\}_i$ be its MFCN projections and components as previously introduced.
For each MFNC component $\Phi_i$\\
- the dimension $d_i$ of ${\mathcal Z}(\Ne_i)$ is finite,
\\
- the minimal projections $Q^i_0,\dots, Q^i_{d_i-1}$ of the abelian algebra ${\mathcal Z}(\Ne_i)$ can be numbered in such a way to form the unique cyclic resolution for 
  $\Phi_i$, i.e. 
 \[
 Z_i=\sum_{m=0}^{d_i-1}Q^i_m
 \qquad \mbox{and} \qquad
\Phi(Q^i_m)=Q^i_{m - 1}
 \]
 (where the subtraction on indices is modulo $d_i$).
\end{prop}

\begin{proof}  Let $Q^i_0,Q^i_1,\dots$ be minimal central projections in $\Ne_i$, then clearly all  $Q^i_m$ are minimal central 
projections in $\Ne$ and we have $\sum_mQ^i_m=Z_i$. Since the restriction of $\Phi_i$ to $\Ne_i$ is a *-automorphism, $\Phi(Q^i_m)=
\Phi_i(Q^i_m)$ is a minimal central projection as well. Put
\[
d_i:=\inf\{ m, \Phi^m(Q^i_0)=Q^i_0\},
\]
then since $\Phi$ preserves the  faithful state $\rho$, $d_i<\infty$. Assume that the projections  are numbered  so that 
\[
Q^i_{d_i-m}=\Phi^m(Q^i_0),\qquad m=0,\dots, d_i-1.
\]
Put $Q^i:=\sum_{m=0}^{d_i-1} Q^i_m$, then obviously $Q^i\in \mathcal Z(\Ne)$ and $\Phi(Q^i)=Q^i$, so that $Q^i\in \mathcal Z$. 
Since also $Q^i\le Z_i$ and $Z_i$ is minimal in $\mathcal Z$, we must have $Q^i=Z_i$.
\end{proof}

We now describe the action of $\Phi_i$ on one component $\Ne_{i}$. 
For simplicity, we drop the index $i$, this corresponds to assuming that there is only one such component, so that $\mathcal Z$ is trivial, $d$ is the period of $\Phi$ and $Q_0,\dots, Q_{d-1}$ is the cyclic resolution of $\Phi$ (as in Definition \ref{def:period2}).

Since $\Ne$ is the range of $E_\Ne$, we may use Lemma \ref{lemma:condexp} to describe its structure.
Let us denote $\Ka_m:=Q_m\Ha$, then there are 
Hilbert spaces $\Ka_{m}^L, \Ka_{m}^R$, $m=0,\dots,d-1$ and unitaries $S_{m}: \Ka_{m}\to\Ka_{m}^L\otimes \Ka_{m}^R$ such that 
\begin{equation}\label{eq:dfs}
\Ne=\bigoplus_{m=0}^{d-1} S_{m}^*(B(\Ka_{m}^L)\otimes I_{m}^R)S_{m}.
\end{equation}
 Here we put $I_m^R=I_{\Ka_{m}^R}$ to simplify notations, we will use a similar notation for $I_{\Ka_{m}^L}$. 
 Let also $\rho_m\in \states(\Ka_{m}^R)$ denote  the states determining  $E_\Ne$, as in Lemma \ref{lemma:condexp} (iii).
The following proposition clarifies some aspects in the structure of the channel $\Phi$ and its action on the DFA.

\begin{prop}\label{prop:dfs} Assume that $\mathcal Z$ is trivial and let the period of $\Phi$ be $d$. Then there are 
\begin{enumerate}
\item[(a)] unitaries $T_m:\mathcal K_m^L\to \mathcal  K^L_{m- 1}$, $m\in{\mathbb Z}_d$,
\item [(b)] quantum channels $\Xi_m: B(\Ka^R_m)\to B(\Ka^R_{m- 1})$, $m\in{\mathbb Z}_d$,
\end{enumerate}
such that for all $m$, 
\begin{enumerate}
\item[(i)] $(\Xi_m)_*(\rho_{m-1}) = \rho_{m}$;
\item[(ii)] $\Xi_{m- (d-1)}\circ\dots \circ \Xi_{m- 1}\circ\Xi_m$ is irreducible and aperiodic;
\item[(iii)] the restriction  $\Phi|_{B(\Ka_m)}$ is a quantum channel $B(\Ka_m)\to B(\Ka_{m- 1})$, determined as
\[
\Phi(S_m^*(A_m\otimes B_m)S_m)=S_{m- 1}^*(T_mA_mT_m^*\otimes \Xi_m(B_m))S_{m- 1};
\]
\item[(iv)] $\Phi$ has a Kraus representation $\Phi(A)=\sum_k V_k^*AV_k$, such that
\[
V_k=\sum_m S_m^*(T_m^*\otimes L_{m,k})S_{m- 1},
\]
where $\Xi_m=\sum_k L_{m,k}^*\cdot L_{m,k}$ is a Kraus representation of $\Xi_m$. 
\end{enumerate}

\end{prop}

\begin{rem}
The results in this proposition are in some points parallel to what discussed in \cite{DFSU} for continuous time Markov semigroups: what is intrinsically different here in our paper is the presence of a supplementary decomposition due to the period, which cannot appear in continuous time.
\end{rem}

\begin{proof}

Let $A_m\in B(\Ka_{m}^L)$. Since  $\Phi(Q_{m}\Ne)=Q_{m- 1}\Ne$,
 we have 
\[
\Phi(S_{m}^*(A_m\otimes I_m^R)S_{m})=S_{m- 1}^*(A_m'\otimes I_{m- 1}^R)S_{m- 1}
\]
 for some $A_m'\in B(\Ka_{m- 1}^L)$ and the map $A_m\mapsto A_m'$ defines a 
*-isomorphism
 of $B(\Ka_{m}^L)$ onto $B(\Ka_{m- 1}^L)$. Hence there is a unitary operator $T_{m}:\Ka_{m}^L\to \Ka_{m- 1}^L$, such that
$A_m'=T_mA_mT_m^*$. Moreover, by the multiplicativity properties  of $\Phi$ on $\Ne$ (see eq. \eqref{MultProp}), we have, for all $B_m\in B(\Ka_m^R)$,
\begin{align*}
\Phi(S_{m}^*(A_m\otimes B_m)S_{m})&=\Phi(S_m^*(A_m\otimes I_{m}^R)S_m)\Phi(S_m^*(I_{m}^L \otimes B_m)S_m)\\
&= \Phi(S_m^*(I_{m}^L \otimes B_m)S_m)\Phi(S_m^*(A_m\otimes I_{m}^R)S_m).
\end{align*}
It follows that $\Phi(S_m^*(I_{m}^L\otimes B_m)S_m)$ is an element in $B(\Ka_{m- 1})$, commuting with all elements in $S_{m- 1}^*(B(\Ka_{m- 1}^L)\otimes I_{m- 1}^R)S_{m- 1}$,
 so that 
\[
\Phi(S_m^*(I_{m}^L \otimes B_m)S_m)= S_{m- 1}^*(I_{m- 1}^L \otimes B'_{m})S_{m- 1}
\]
 for some $B_m'\in B(\Ka_{m- 1}^R)$. 
 It is clear that $B_m\mapsto B_m'$ defines a quantum channel $\Xi_m: B(\Ka^R_m)\to B(\Ka^R_{m- 1})$. 
Putting all together proves (iii). 

To see (ii), let $\tilde \Xi_m=\Xi_{m- (d-1)}\circ\dots \circ \Xi_{m- 1}\circ\Xi_m$ be the given composition and let $R_m\in B(\Ka^R_m)$ be a  projection that is either fixed or decoherence-free for  $\tilde \Xi_m$. Then 
$S_m^*(I_m^L\otimes R_m)S_m$ is in $\Ne$, so that $R_m$ must be trivial. 

Further, for (i) note that by Corollary \ref{coro:range},  $E_\Ne$ commutes with  $\Phi$.
 For  $B_m\in B(\Ka^R_m)$, we have by Lemma \ref{lemma:condexp}
\[
\Phi\circ E_\Ne(S_m^*(I_m^L\otimes B_m)S_m)=\Tr[\rho_m B_m]\Phi(Q_m)=\Tr[\rho_m B_m]Q_{m- 1}
\]
and
\[
E_\Ne\circ \Phi(S_m^*(I_m^L\otimes B_m)S_m)=E_\Ne(S_{m-1}^*(I_{m- 1}^L\otimes \Xi_m(B_m))S_{m- 1})=\Tr[\rho_{m- 1} \Xi_m(B_m)]Q_{m- 1},
\]
so that (i) holds.

Finally, let $\Phi=\sum_k V_k^*\cdot V_k$ be any Kraus representation of $\Phi$. Then we have
\[
\Phi(A)=\sum_{m,n=0}^{d-1}\Phi(Q_mAQ_n)=\sum_{m,n=0}^{d-1} Q_{m- 1}\Phi(Q_mAQ_n)Q_{n- 1},
\]
so that we may assume that $V_k=\sum_m V_{k,m}$, with $V_{k,m}=Q_mV_kQ_{m- 1}$, for all $k$ and $m$. Moreover, for each $m$,
 $\sum_k V_{k,m}^*\cdot V_{k,m}$ is a Kraus representation of the restriction $\Phi|_{B(\Ka_m)}$. 

Let $\Xi_m=\sum_l K_{m,l}^*\cdot K_{m,l}$ be a minimal Kraus representation. It follows from (iii) that 
\[
\Phi|_{B(\Ka_m)}=\sum_l S_{m- 1}^*(T_m\otimes K_{m,l}^*)S_m\cdot S_{m}^*(T_m^*\otimes K_{m,l})S_{m- 1}
\]
is another Kraus representation of $\Phi|_{B(\Ka_m)}$, hence there are some $\{\eta_{k,l}^j\}$ such that 
$\sum_i \eta^j_{i,k}\bar{\eta}_{i,l}^j=\delta_{k,l}$ and 
\[
V_{k,m}=\sum_l\eta^j_{k,l} S_{m}^*(T_m^*\otimes K_{m,l})S_{m- 1}= S_{m}^*(T_m^*\otimes L_{m,k})S_{m- 1},
\]
where $L_{m,k}:=\sum_l\eta^m_{k,l} K_{m,l}$, this proves (iv).
\end{proof}

Note that  by identifying 
\[
\Ha=\bigoplus_m \Ka_m\simeq \sum_m \Ka_m\otimes  {\mathbb C} |m\>
\] 
and 
\[
\Ka:=\bigoplus_m \Ka_m^L\otimes \Ka^R_m \simeq \sum_m \Ka^L_m\otimes \Ka_m^R\otimes   {\mathbb C} |m\>,
\]
\eqref{eq:dfs} can be written as
\[
\Ne= S^*(\sum_{m=0}^{d-1} B(\Ka_m^L)\otimes I_m^R\otimes \ketbra{m}{m})S, 
\]
where $S: \Ha\to \Ka$  is a unitary given as $S=\sum_m S_m\otimes \ketbra{m}{m}$. We will also use the notation
\[
\Ka^R:=\bigoplus_m\Ka^R_m\simeq \sum_m\Ka^R_m\otimes  {\mathbb C} |m\> 
\]
and put $I^R:=I_{\Ka^R}$. 
We are now ready to describe the subalgebra $\Fe$. In the following proposition, we keep the notations of Proposition \ref{prop:dfs}.
We can see the next step as an improvement of Lemma \ref{lemma:condexp} applied to the fixed points domain $\Fe$: we can give a more detailed description of $\Fe$ and construct some of the mathematical objects appearing in the lemma using the items introduced in Proposition \ref{prop:dfs}. Going to the predual vision, we can then consider the structure of the invariant states and, finally, we can present the action of the channel on the subsystems associated with the central projections of $\Fe$.

\begin{thm}\label{prop:fixed_multi}
Assume that $\mathcal Z$ is trivial and let the period of $\Phi$ be $d$.
Let us denote 
\[
\tilde T_m:\Ka_0^L\to \Ka_{m}^L, \quad \tilde T_m:=T_{m+1}\dots T_{d-1}T_0, \ m=0,\dots,d- 2; \quad \tilde T_{d-1}:=T_0
\]
and let $T: \Ka_0^L\otimes \Ka^R\to \Ka$ be the unitary 
defined as
\[
T=\sum_{m=0}^{d-1} \tilde T_{m}\otimes I_m^R\otimes \ketbra{m}{m}.
\]
\begin{enumerate}

\item[(i)] The operator $\tilde T_{0}\in  \mathcal U(\Ka_0^L)$ has a discrete spectrum. Let $R_j$ be its minimal spectral projections and let $\mathcal L_j:=R_j\Ka^L_0$, then 
\[
\Fe=S^*T(\bigoplus_j B(\mathcal L_j)\otimes I^R)T^*S;
\]
\item[(ii)] Let $\sigma_j\in \states(\Ka^R)$ be the faithful normal states corresponding to $E_\Fe$ as in Lemma \ref{lemma:condexp} (iii) and (iv). Then 
\[
\sigma_j\equiv \sigma:=\frac 1d\sum_{m=0}^{d-1} \rho_m\otimes \ketbra{m}{m},\qquad \forall j;
\]

 \item[(iii)] Invariant states $\xi\in \mathfrak S(\Ha)$ for $\Phi$ are precisely those of the form
\[
\xi=  S^*T\left(\omega\otimes \sigma \right)T^*S,
\]
where $\omega=\sum_{j}\lambda_j\omega_j\otimes \ketbra{j}{j}$ for some probabilities $\{\lambda_j\}$ and states $\omega_j\in \mathfrak S(\mathcal L_j)$.

\item[(iv)] Let $P_j:=S^*T(R_j\otimes I^R)T^*S$ be the minimal central projections in $\Fe$. The restrictions $\Phi|_{B(P_j\Ha)}$
 have the form
 \[
\Phi|_{B(P_j\Ha)}(S^*T(A_j\otimes B)T^*S)= S^*T(A_j\otimes \Psi_j(B))T^*S,\quad A_j\in B(\mathcal L_j), B\in B(\Ka^R),
 \]
 where $\Psi_j$ are irreducible quantum channels on $B(\Ka^R)$. Moreover, all $\Psi_j$ coincide on block-diagonal elements of the form 
  $\sum_m B_{mm}\otimes \ketbra{m}{m}$  and we have
 \[
\Psi_j(\sum_m B_{mm}\otimes \ketbra{m}{m})=\sum_m \Xi_m(B_{mm})\otimes \ketbra{m- 1}{m-1}.
 \]
 In particular, for all $j$, $\Psi_j$ has period $d$,  $\Ne(\Psi_j)=span\{I_m^R\otimes \ketbra{m}{m}, m=0,\dots, d-1\}$ and $\sigma$
 of (ii) is the unique invariant state.
\end{enumerate}
\end{thm}

\begin{proof} Since $\mathcal F\subseteq \Ne$, we may apply Proposition \ref{prop:dfs}. It can be easily checked that 
an element of $\Ne$ is in $\Fe$ if and only if it is of the form
\[
S^*T(A\otimes I^R)T^*S
\]
with $A\in\mathcal A:=\{\tilde T_{0}\}'\cap B(\Ha_0^L)$. Note that the commutant 
$\mathcal A':=\{\tilde T_{0}\}''\cap B(\Ha_0^L)=\mathcal Z(\mathcal A)$ is abelian. Further, we have
$\mathcal F\simeq \mathcal A$ and since $\mathcal F$ is 
 atomic, $\mathcal A$ must be such as well, so that $\{\tilde T_{0}\}''\cap B(\Ha_0^L)$ must be discrete. This proves (i).

By Lemma 
\ref{lemma:condexp}, there are some states $\sigma_j\in \states(\Ka^R)$ such that
\begin{equation}\label{eq:sigma}
E_\Fe(S^*T(R_j\otimes B)T^*S)=\Tr [\sigma_j B]P_j,
\end{equation}
where $B\in B(\Ka^R)$ and $P_j:=S^*T(R_j\otimes I^R)T^*S$ are the minimal central projections in $\Fe$. Moreover, 
since $E_\Fe$ is given by \eqref{eq:limit} and satisfies \eqref{eq:condexp_min}, we see that a state  $\xi$ is invariant for $\Phi$ if and only if it is invariant for $E_\Fe$. Consequently, by Lemma \ref{lemma:condexp} (iv),
any state of the form 
$\psi=T^*S(\omega_j\otimes \sigma_j)S^*T$ with $\omega_j\in \states(\mathcal L_j)$ is an invariant state for $\Phi$. It follows that
for any $m=0,\dots d-1$,
\begin{align*}
\Tr [\sigma_j(I_m\otimes\ketbra{m}{m})]&=\Tr [\psi S^*T(R_j\otimes I_m\otimes\ketbra{m}{m})T^*S]\\
&=
\Tr [\Phi_*(\psi)S^*T(R_j\otimes I_m\otimes\ketbra{m}{m})T^*S]\\
&=
\Tr [\psi\Phi(S^*T(R_j\otimes I_m\otimes\ketbra{m}{m})T^*S)]\\
&=\Tr[\psi S^*T(R_j\otimes I_{m- 1}\otimes\ketbra{m- 1}{m- 1})T^*S]\\
&=\Tr [\sigma_j(I_{m- 1}\otimes\ketbra{m- 1}{m- 1})]
\end{align*}
so that $\Tr [\sigma_j(I_m\otimes\ketbra{m}{m})]=1/d$. Let now $B=\sum_{m,n} B_{mn}\otimes \ketbra{m}{n}\in B(\Ka^R)$.
Since $E_\Ne\in \mathbf S$, we obtain from \eqref{eq:condexp_min} that also
$E_\Fe\circ E_\Ne=E_\Ne\circ E_\Fe=E_\Fe$. Using Lemma \ref{lemma:condexp} (ii) for $E_\Ne$, we get 
 \begin{align*}
E_\Fe(S^*T(R_j\otimes B)T^*S)&= \sum_m E_\Fe\circ E_\Ne(Q_mS^*T(R_j\otimes B)T^*SQ_m)\\
&=
\sum_m E_\Fe\circ E_\Ne(S^*_m(\tilde T_{m}R_j\tilde T_{m}^*\otimes B_{mm})S_m)
\\
\mbox{\tiny (by Lemma \ref{lemma:condexp} (iii)) }&= \sum_m \Tr[ \rho_m B_{mm}]E_\Fe(S^*_m(\tilde T_{m}R_j\tilde T_{m}^*\otimes I_m^R)S_m)\\
&=
\sum_m \Tr[\rho_m B_{mm}]E_\Fe(S^*T(R_j\otimes I_m^R\otimes \ketbra{m}{m})T^*S)=\frac1d\sum_m\Tr [\rho_m B_{mm}] P_j,
 \end{align*}
 where the last equality follows from \eqref{eq:sigma} and the previously obtained equality $\Tr[\sigma_j I_m\otimes\ketbra{m}{m}]=1/d$. 
 This and \eqref{eq:sigma} prove (ii). 
 
 Point (iii) now directly follows from Lemma \ref{lemma:condexp} (iv).

Finally, we prove (iv). We see by the multiplicativity properties of $\Phi$ on $\Ne$ that $\Phi(B(P_j\Ha))\subseteq B(P_j\Ha)$ and that the restrictions have the given form with some quantum channel $\Psi_j$ on $B(\Ka^R)$. Since any fixed point of $\Psi_j$ is related to a fixed point of $\Phi$, we can see that it must be trivial, so that $\Psi_j$ are irreducible. For any $B_{m}\in B(\Ka_m^R)$, we have
 by Proposition \ref{prop:dfs}, 
\begin{align*}
\Phi(S^*T(R_j\otimes B_m\otimes \ketbra{m}{m})T^*S&=\Phi(S_m^*(\tilde T_{m-1}R_j\tilde T_{m-1}^*\otimes B_m)S_m)\\
&=S_{m- 1}^*(\tilde T_{m}R_j\tilde T_{m}^*\otimes \Xi_m(B_m))S_{m- 1})\\
&=S^*T(R_j\otimes \Xi_m(B_m)\otimes \ketbra{m-1}
{m- 1})T^*S.
\end{align*}
It follows that $\Psi_j (B_m\otimes \ketbra{m}{m})=\Xi_m(B_m)\otimes \ketbra{m-1}{m- 1}$ and 
$I_m^R\otimes \ketbra{m}{m}\in \Ne(\Psi_j)$ for all $m$ and $j$. Hence any minimal projection in $\Ne(\Psi_j)$
 must be of the form $Q\otimes \ketbra{m}{m}$ for some $m=0,\dots,d-1$ and some projection $Q\in B(\Ka^R_m)$. But then it easily follows that $I_m\otimes Q$   is in $\Ne$, so that we must have $Q=I_m^R$. 
Finally, the fact that $\sigma$ is an ivariant state for $\Psi_j$ follows easily from (iii).
\end{proof}

{\bf Conclusions.} 
We are aware that the contents of this subsection are technical and the relations between different representations and decompositions are intricate, in particular for  a reader who is not involved in similar research topics.
For a full comprehension, it can be useful to insert it in the surrounding literature.
As we already mentioned in the Introduction, the results of this section include sometimes a revision and improvements or generalizations of different previous studies.
The structure of the fixed points domain has already been investigated and one can find various papers in last two decades, see for instance  \cite{AGG,BN,BJKW,capa2015,JencovaPetz} and references therein.
For the structure of the DFA, there is some interest growing from different fields and we could improve its description in Theorem \ref{prop:fixed_multi}. We can underline that here we study a dual version in infinite dimension of 
the decomposition appearing in \cite[Theorem 8]{WPG} and \cite{WolfLN};
further, this section includes a generalization, in discrete time version (which has a richer structure, due to period) of \cite{DFSU}, without the need of imposing atomicity condition.


\section{Application to open quantum walks}\label{section:OQRWs}

In this section we discuss an important type of quantum channels.

Let $\Ha= \oplus_{i\in V} \ha_i$, where $V$ is a countable set of vertices and $\ha_i$ are separable Hilbert spaces. Note that we may express $\Ha$ as $\Ha=\sum_{i\in V} \ha_i\otimes |i\>$. 
An open quantum random walk (OQRW) (\cite{Attal}) is a completely positive trace preserving map $\mathfrak M$ on the space $B(\Ha)_*$ of trace-class operators, of the form  
\[
\mathfrak M: \rho \mapsto \sum_{i,j} V_{i,j}\rho V_{i,j}^*,
\]
where $V_{i,j}=L_{i,j}\otimes |i\>\<j|$ and $L_{i,j}$ are bounded operators $\ha_j\to \ha_i$ satisfying
\begin{equation}\label{eq:tracepreserving}
\sum_{i\in V} L_{i,j}^*L_{i,j}= I_j, \qquad \forall j\in V.
\end{equation}

Put $\Phi=\mathfrak M^*$, then $\Phi$ is a quantum channel. Note that any operator $A\in B(\Ha)$ can be written as
\[
A=\sum_{i,j\in V} A_{i,j}\otimes |i\>\<j|,
\]
where $A_{i,j}$ is a bounded operator $\ha_j\to \ha_i$, and the action of $\Phi$ has the form
\[
\Phi(A)=\sum_{j} \sum_iL_{i,j}^*A_{i,i}L_{i,j}\otimes \ketbra{j}{j}.
\]
This family of quantum channels has recently become quite popular and have been extensively studied (see \cite{BBP,capa2015,DM,GFY,LS,SP2019}).
Here we want to investigate the structure of the DFA associated with an OQRW: we obtain some results in the general case and then expound some particular remarkable classes. 
Finally we go exploring a non positive recurrent family of models considering homogeneous OQRWs on the group $V={\mathbb Z}$.

We next characterize the multiplicative domain $\Me$ and the decoherence-free subalgebra $\Ne$ of $\Phi$ by the transition operators $L_{i,j}$.

To obtain $\Ne$, we invoke the notation of \cite{capa2015}. For $i,j\in V$, let  $\mathcal P_n(i,j)$ be the set of all paths 
\[
\pi=(i_0=i,i_1,\dots,i_n=j)
\]
from $i$ to $j$ of length $n$. For each $\pi \in  \mathcal P_n(i,j)$, we define the operator $L_\pi:\ha_i\to \ha_j$ as
\[
L_\pi=L_{j,i_{n-1}}L_{i_{n-1},i_{n-2}}\dots L_{i_1,i}.
\]

\begin{prop}\label{prop:oqrw-m&n} Let $\Phi$ be an OQRW. 
\\
1. $A\in \Me$ if and only if 
\begin{equation}\label{eq:oqrwmdiag}
A_{i,i}L_{i,j}L_{k,j}^*= L_{i,j}L_{k,j}^*A_{k,k}, \quad \forall i,j,k
\end{equation}
\begin{equation}\label{eq:oqrwmoff}
\mbox{and } \qquad A_{l,i}L_{i,j}=0= L^*_{l,j}A_{l,i}, \quad \forall i,j,l\in V,\ i\ne l
\end{equation}
2.
$A\in \Ne$  if and only if for all $i,j,k,l\in V$, $l\ne i$ and $n\in \mathbb N$,
\begin{equation}\label{eq:oqrmndiag}
A_{i,i}L_\pi L_{\pi'}^*= L_\pi L_{\pi'}^*A_{k,k}, \quad \forall \pi\in\mathcal P_n(j,i), \ \pi' \in  \mathcal P_n(j,k)
\end{equation}
\begin{equation}\label{eq:oqrmnoff}
\mbox{and } \qquad A_{l,i}L_{i,j}=0=L_{l,j}^*A_{l,i}.
\end{equation}

\end{prop}

\begin{proof}
1. It is easy to see from Proposition \ref{prop:multiplicative} that $A\in \Me$ if and only if $A$ commutes with all operators of the form $L_{i,j}L_{k,j}^*\otimes |i\>\<k|$, $i,j,k\in V$. This is equivalent to (\ref{eq:oqrwmdiag}), together with
\begin{equation}\label{eq:oqrwmoffff}
A_{l,i}L_{i,j}L_{k,j}^*=0= L_{k,j}L_{l,j}^*A_{l,i}, \quad \forall i,j,k,l\in V,\ l\ne i
\end{equation}
It is clear that (\ref{eq:oqrwmoff}) implies (\ref{eq:oqrwmoffff}). For the converse, 
 multiply the first equality of (\ref{eq:oqrwmoffff}) by $L_{k,j}$ from the right and sum over $k\in V$, then (\ref{eq:tracepreserving}) implies the first equality of (\ref{eq:oqrwmoff}). The second equality is proved similarly.
 
2.
Since the Kraus operators of $\Phi^n$ are operators of the form $L_\pi\otimes |i\>\<j|$ for $\pi\in \mathcal P_n(j,i)$, 
the second statement can be proved exactly as the previous one. 
\end{proof}

Due to the characterization in the previous proposition, we can deduce a decomposition of the decoherence-free algebra in block diagonal and block off-diagonal operators.

\begin{coro}\label{diag-off-diag}
$\Ne=\Ne_D\oplus \Ne_{OD}$ where:
\begin{eqnarray*}
&&\kern-10pt 
\Ne_D := \{A=\sum_{i\in V} A_{i}\otimes |i\>\<i|, A\in\Ne\} 
\\
&&=\{\sum_{i\in V} A_{i}\otimes |i\>\<i| \,:\, A_{i}L_\pi L_{\pi'}^*= L_\pi L_{\pi'}^*A_{k},  \forall i,k\in V, 
\forall (\pi,\pi')\in \cup_{j,n}(\mathcal P_n(j,i)\times \mathcal P_n(j,k))  \}
\\
&&\kern-10pt
\Ne_{OD} :=  \{A=\sum_{i\neq j\in V} A_{i,j}\otimes |i\>\<j|, A\in\Ne\}
\\
&&=\{\sum_{i\neq j\in V} A_{i,j} \otimes |i\>\<j| \,:\, A_{i,j}L_{j,l}=0=L_{i,l}^*A_{i,j},\,  \forall i\ne j,l\in V \}.
\end{eqnarray*}
\end{coro}

Assume that  $\Ne_{OD}$ is non-trivial, so that there is some $0\ne A\in \Ne_{OD}$. Since $\Phi(A)=0$,   $A\in \Ne\cap \ker \Phi$  and clearly also  $A^*A\in \Ne\cap \ker \Phi$.  The  block-diagonal part
$(A^*A)_D:=\sum_{i\in V} (A^*A)_{i,i}\otimes |i\>\<i|$ is  a nonzero positive operator in 
$$\Ne\cap \ker \Phi\cap \{\mbox{block diagonal operators}\}=\Ne_D\cap \ker \Phi.$$
Summing up, we deduce
$$
\Ne_{OD}\neq \{0\}
\quad\Rightarrow \quad
\Ne\cap \ker \Phi\neq\{0\}
\quad \Leftrightarrow \quad
\Ne_D\cap \ker \Phi\neq\{0\}.
$$

\begin{coro}\label{coro:diag} If $\Phi$ admits a faithful normal invariant state, then $\Ne=\Ne_D$.
In the general case, put $\mathcal W_i:=\cap_{j\in V} \mathrm{Range} (L_{i,j})^\perp$, then $\Ne=\Ne_D$ if and only if there is at most one index  $i\in V$, such that $\mathcal W_i\ne \{0\}$.
\end{coro}

\begin{proof} As shown above, if $\Ne_{OD}\ne \{0\}$, then also $\Ne_D\cap \ker \Phi\neq\{0\}$ and since this is a von Neumann algebra, $\Ne_D\cap \ker \Phi$ must contain a   non trivial projection. This is clearly not possible if $\Phi$ admits a faithful normal invariant state, since then $\Phi$ is faithful and there can be no projections in $\ker\Phi$.
The general case is clear from Corollary \ref{diag-off-diag}.

\end{proof}

\subsection {Homogeneous OQRWs}

An OQRW is called homogeneous if $V$ is an abelian  group, $\ha_i=\ha$ does not depend on $i$ and transition operators are translation invariant, i.e. 
$L_{i,j}=L_{i+n,j+n}=:L_{j-i}$ for any $i,j,n\in V$.  We can define the local operator $\mathcal L$, acting on $B(\mathfrak h)_*$ as 
$$
\mathcal L(\rho) = \sum_{k\in V} L_k \rho L_k^*.
$$
Let $\rhoinv\in \states(\ha)$ be an invariant state for $\mathcal L$. If $V$ is finite, then
\[
\frac 1{|V|} \sum_i \rhoinv\otimes \ketbra{i}{i}
\]
is a normal invariant state for $\Phi$, which is faithful iff $\rhoinv$ is. If $V$ is infinite, we can only obtain an invariant weight in this way, given as 
$$
\omega(\sum x_{ij}\otimes |i\>\<j|) := 
\sum_j \Tr [\rhoinv x_{jj}],
$$
for all positive $x=\sum x_{ij}\otimes |i\>\<j|$ in $B(\Ha)$.
In particular, if $\Phi$ is irreducible, the invariant states must be translation invariant and hence there are no invariant states if $V$ is infinite, \cite[Prop. 9.3.]{capa2015}. 

We will consider the nearest neighbor case with $V=\mathbb Z$  (or $V=\mathbb Z_d$) and $L_{i-1,i}=L_-$, $L_{i+1,i}=L_+$, all the other $L_{i,j}=0$.
An immediate application of Proposition \ref{prop:oqrw-m&n}  will give us the following.

\begin{coro}\label{coro-OQRWs-M} Let $\Phi$ be a homogeneous nearest neighbor OQRW on  $\mathbb Z$ (or $\mathbb Z_d$). Then $A\in \Me$ if and only if 
\begin{eqnarray*}
A_{i,i}L_{+}L_{-}^*= L_{+}L_{-}^*A_{i-2,i-2}, 
\quad
A_{i-2,i-2}L_{-}L_{+}^*= L_{-}L_{+}^*A_{i,i},
\quad
A_{i,i}\in\{|L_{+}^*|, |L_{-}^*| \}'
\quad \forall i,
\end{eqnarray*}
and 
\begin{equation}\label{eq:homoqrwmoff}
A_{ik}L_{-}=A_{ik}L_{+}=L_{-}^*A_{ik}=L_{+}^*A_{ik}=0, \quad \forall i,k\in V,\ i\ne k
\end{equation}
In particular, when at least one transition operator is invertible, $\mathcal M$ contains only block-diagonal operators.
\end{coro}

\subsection{Examples}

We will consider three examples of open quantum random walks and describe their decoherence free algebras. As we will see in the first example, the action of any quantum channel on a cyclic component of $\Ne$  is decribed by an OQRW of a certain  type. The second example is a homogeneous  OQRW with two vertices and finite dimensional local spaces. In the last example, we describe the decoherence-free algebra for a homogeneous OQRW  without a faithful normal invariant state.

\subsubsection{A cyclic shift OQRW}

We consider an OQRW with $V=\mathbb Z_d$ and $\ha_0\simeq\ha_1\simeq\dots\simeq\ha_{d-1}$. 
$L_{i,i- 1}=:U_i$ be a unitary $\ha_{i- 1}\to \ha_i$, $i=0,\dots ,d-1$,  $L_{i,j}=0$ for $i-j\ne 1$
(addition and subtraction on indices are in $\mathbb Z_d$).
We can explicitly write the action of $\Phi$ and its preadjoint as
\[
\Phi(A)=\sum_{i=0}^{d-1}U_i^*A_{i,i}U_i\otimes \ketbra{i- 1}{i- 1},
\qquad  \Phi_*(\rho)=\sum_{i=0}^{d-1}U_i\rho_{i- 1,i- 1}U_i^*\otimes \ketbra{i}{i}
\]
where $A=\sum_{i,j} A_{i,j} \otimes \ketbra{i}{j}$ and $\rho=\sum_{k,l}\rho_{k,l}\otimes|k\>\<l|$. It is clear from this expression that the  fixed points of $\Phi$ 
are precisely the  block diagonal operators such that 
\[
U_i^*A_{i,i}U_i=A_{i- 1,i- 1},\qquad i=0,\dots,d-1.
\]
Putting $\tilde U_i:= U_i\dots U_1:\ha_{0}\to \ha_i$, $i=1,\dots, d-1$, $\tilde U_{0}:=I_{\ha_0}$ and 
$\tilde U:=U_0\tilde U_{d-1} \in B(\ha_0)$, we obtain that  
 $$
\Fe=\{\sum_{i=0}^{d-1} \tilde U_{i}A_0\tilde U_{i}^*\otimes |i\>\<i|, \, A_0\in B(\ha_0) \,s.t.\, [A_0,\tilde U]=0\}.
$$
It follows that $\Fe$ is a von Neumann algebra isomorphic to $\{\tilde U\}'\cap B(\ha_0)$ and hence we always have $\Fe\subseteq \Ne$ (Proposition \ref{prop:fixed}).
Similarly, the invariant normal states have the form
\[
\rho=\sum_{i=0}^{d-1} \tilde U_{i}\rho_0\tilde U_{i}^*\otimes |i\>\<i|,\qquad \rho_0\in \states(\ha_0),\quad
 \tilde U\rho_0\tilde U^*=\rho_0.
\]
It follows that  normal invariant states for $\Phi$ are obtained from normal invariant states for the unitary conjugation $\tilde U\cdot \tilde U^*$.

Due to Corollary  \ref{coro:diag},  $\Ne=\Ne_D$ is block diagonal, moreover, since for any path $\pi$ of length $n$, the operator $L_\pi$ is nonzero if and only if $\pi=(i,i+ 1,\dots,i+ n)\in \mathcal P_n(i,i+ n)$, we can see from Corollary \ref{diag-off-diag} that 
 $\Ne$ consists of all block diagonal operators, i.e. 
$$
\Ne=\{\sum_{i=0}^{d-1} A_i\otimes |i\>\<i|,\, A_i\in B(\ha_i),\, i=0,\dots,d-1\}
$$
with minimal central projections $I_{\ha_i}\otimes |i\>\<i|$, $i=0,\dots,d-1$.

It is then clear that $\Phi$ has a unique MFNC component, i.e. $\mathcal Z=\mathcal Z(\Fe)\cap \mathcal Z(\Ne)$ is trivial, with period $d$. It can be instructive to compare the above example to the results of Section \ref{sec:reducible}, when a faithful normal invariant state exists. 
In the notations of  Proposition \ref{prop:dfs}, we have $\Ka_m^L=\ha_m$, $\dim(\Ka_m^R)=1$ and 
$T_m=U_m^*$, $m=0,\dots,d-1$.  
Moreover, we see that the obtained form of  $\Fe$ and of invariant normal states corresponds to the results of Theorem 
\ref{prop:fixed_multi}, where here $\Ka^R\simeq \mathbb C^d$ and  $\sigma$ is the tracial state on  $B(\mathbb C^d)$.
In fact, we can observe the following result.

\begin{prop}\label{prop:shift} Let $\Phi$ be a quantum channel admitting a faithful normal invariant state, with a unique MFNC component. Then there exists an OQRW $\Psi$ of the above form, such that $\Ne(\Phi)\simeq \Ne(\Psi)$ and $\Phi|_{\Ne(\Phi)}\simeq \Psi|_{\Ne(\Psi)}$.

\end{prop}

\subsubsection{A homogeneous OQRW  with generalized Pauli operators}

Let $V=\mathbb Z_2$ and $\ha_0=\ha_1=\ha$. Let $L_{00}=L_{11}=\sqrt{\alpha} U_0$, $L_{01}=L_{10}=\sqrt{1-\alpha}U_1$, with $\alpha\in (0,1)$ and $U_0,U_1$ unitaries on $\ha$. Explicitly, $\Phi$ acts as
\[
\Phi(A)=\left[\alpha(U_0^*A_{00}U_0)+(1-\alpha)U_1^*A_{11}U_1\right]\otimes|0\>\<0|+\left[(1-\alpha)U_1^*A_{00}U_1+\alpha U_0^*A_{11}U_0\right]\otimes|1\>\<1|.
\]
Assume $d:=\dim(\ha)<\infty$, then $(2d)^{-1} I_\ha\otimes I_2$ is a faithful invariant state for $\Phi$. We next investigate the fixed points and decoherence free subalgebra in the case when $U_0$ and $U_1$ are generalized Pauli operators described below.

Let $\{|j\>,\ j=0,\dots,d-1\}$ denote a fixed ONB in $\ha$ and let $\oplus$ be addition modulo $d$. Put $\omega=e^{i2\pi/d}$ and define
 the operators $Z$ and $X$ as
\begin{align*}
Z|j\>&=\omega^j|j\>\\
X|j\>&=|j\oplus 1\>
\end{align*}
Then $Z$ and $X$ are unitaries satisfying the commutation relation
\[
ZX=\omega XZ.
\]
Let us also denote
\[
W(p)=Z^pX^{-p},\qquad p\in \mathbb Z,
\]
then $W(p)$ satisfy the relations
\begin{equation}\label{eq:weyl}
W(p)W(q)=W(q)W(p)=\omega^{pq}W(p+q).
\end{equation}

Let $\Phi$ be an OQRW as above, with 
$$U_0=Z,  \qquad U_1=X.
$$ 
We first find the fixed point subalgebra of $\Phi$, this can be done using Proposition \ref{prop:fixed}. We see that
\[
\Fe=\{Z\otimes |0\>\<0|,\ Z\otimes |1\>\<1|,\ X\otimes |0\>\<1|, X\otimes |1\>\<0|\}'
\]
and from this, we get
\begin{equation}\label{eq:fixed_pauli}
\Fe=\{\left(\begin{array}{cc} A & 0\\ 0 & XAX^*, \end{array}\right),\ A\in \{Z,X^2\}'\}.
\end{equation}
The condition $A\in \{Z,X^2\}'$ implies that $A$ is  diagonal in the basis $\{|j\>\}$ and   
\[
A=X^2A(X^*)^2\implies \sum_j a_j|j\>\<j| = \sum_j a_j |j\oplus 2\>\<j\oplus 2|,
\]
so that $a_j=a_{j\oplus 2}$ for $j=0,\dots,d-1$. 

Assume now that $d$ is odd. Then it follows that $a_j=a_0$ for all $j$, so that $\Fe$ is trivial.
 Hence, in this case, $\Phi$ is irreducible. Put $W=W(1)=ZX^*$, then 
\[
Z^*WZ=X^*WX=\omega W
\]
It follows that $\Phi(W\otimes I_2)=\omega(W\otimes I_2)$, so $\tilde W:= W\otimes I_2$ is an eigenvector related to the 
peripheral eigenvalue $\omega$. The eigenvalues of $W$ are $\omega^k$, $k=0,\dots,d-1$, each with an eigenvector $|x_k\>$. 
Hence the period of $\Phi$ is $d$ and we have the  cyclic decomposition 
\[
\{Q_m=|x_m\>\<x_m|\otimes I_2,\ m=0,\dots, d-1\}. 
\]
By the results of Subsection \ref{Subsection:IQC}, $\Ne$ is spanned by $\{Q_0,\dots, Q_{d-1}\}$.

We next turn to the more interesting case when  $d$ is even. Put $q=d/2$.
Then \eqref{eq:fixed_pauli} holds, with $A=a_+P_++ a_-P_-$, where $a_+,a_-\in \mathbb C$ and
\[
P_+=\sum_{k=0}^{q-1}|2k\>\<2k|,\qquad P_-=\sum_{k=0}^{q-1}|2k+1\>\<2k+1|.
\]
So $\Fe$ is isomorphic to the abelian algebra spanned by these two projections.
Note that we have $XP_+X^*=P_-$, $XP_- X^*=P_+$, so that we may write
\[
\Fe=\mathrm{span}\{\tilde P_+:=\left(\begin{array}{cc} P_+ & 0\\ 0 & P_- \end{array}\right),\tilde P_-:=\left(\begin{array}{cc} P_- & 0\\ 0 & P_+ \end{array}\right)\}.
\]
Let us compute $\Ne$ using Proposition \ref{prop:oqrw-m&n}. Note first that by the commutation relations, we have for $\pi\in \mathcal P_n(i,j)$, 
\[
L_\pi=xZ^{n-l}X^l,
\]
where $x\ne 0$ is some constant  and $l\in \mathbb N$ is even if and only if  $i=j$. It follows that if $\pi\in \mathcal P_n(j,i)$, 
$\pi'\in \mathcal P_n(j,k)$, we have
\[
L_\pi L_{\pi'}^*=y Z^{p}X^{-p}=y W(p),
\]
where $0\ne y\in \mathbb C$ and  $|p|$ is even iff $k=i$. Since all $L_{i,j}$ are (nonzero) multiples of unitary operators, we must have $\Ne_{OD}=\{0\}$. 
From the conditions on the diagonal blocks, we obtain that $A_{i,i}$ must commute with $W(p)$ for all even $|p|$ and 
$A_{i,i}=W(p)^*A_{j,j}W(p)$ for all $|p|$ odd if $i\ne j$. Using \eqref{eq:weyl}, we obtain that 
\[
\Ne=\{\left(\begin{array}{cc} A & 0 \\ 0 & WAW^*\end{array}\right),\ A\in \{W(2)\}'\}.
\]
It follows that $\Ne$ is isomorphic to the algebra $\{W(2)\}'$. One can see by \eqref{eq:weyl} and $d=2q$ that $W(2)^q=I$, so that the eigenvalues of $W(2)$ are the $q$-th roots of unity, that is, $\mu_m=\omega^{2m}$, $m=0,\dots,q-1$. Let us denote
\[
|m,+\>:=\frac1{\sqrt{q}} \sum_{l=0}^{q-1}\omega^{2l(m-l+1)}|2l\>,\qquad |m,-\>:=\frac1{\sqrt{q}} \sum_{l=0}^{q-1}\omega^{2l(m-l)}|2l\oplus 1\>
\]  
then one can check that
\[
Q_m:=\ketbra{m,+}{m,+}+\ketbra{m,-}{m,-}
\]
is the eigenprojection corresponding to the eigenvalue $\mu_m$. Since $W$ commutes with $W(2)$ by \eqref{eq:weyl}, 
 we have $W Q_mW^*= Q_m$, so that  the center of $\Ne$ is spanned by the projections  
\[
\tilde Q_m:= Q_m\otimes I_2,\qquad m=0,\dots,q-1.
\] 
Further, it is easily checked that for $m=0,\dots,q-1$, we have
\[
Z|m,+\>= |m\oplus_q 1,+\>,\qquad Z|m,-\>= \omega |m\oplus_q 1,-\>
\]
and
\[
X|m,+\> = |m\oplus_q 1,-\>,\qquad X|m,-\>=\bar \omega^{2(m+1)}|m\oplus_q1, +\>.
\]
Since  the action of $\Phi$ on elements of $\Ne$ has the form
\[
\Phi\left(\begin{array}{cc} A & 0 \\ 0 & WAW^*\end{array}\right)=\left(\begin{array}{cc} Z^*AZ & 0 \\ 0 & X^*AX\end{array}\right),
\]
 we obtain  $\Phi(\tilde Q_m)=\tilde Q_{m\ominus_q 1}$.
It follows that there is a unique cycle of length $q$ and consequently only one MFNC component $\Ne_{[1]}=\Ne$, with period $q$.  

We will identify the objects described in Section \ref{sec:reducible} for this special case. 
We have  $\mathcal K_m^L= Q_m\ha\simeq \mathbb C^2$, $\Ka_m^R=\mathbb C^2$ and $\Ka^R=\sum_m \mathbb C^2\otimes \ketbra{m}{m}\simeq \ha$. Put $S_m=I_m^L\otimes |0\>\<0|+W^*|_{\Ka^L_m}\otimes |1\>\<1|$, $m=0,\dots,q-1$, then we have
\[
\Ne=\oplus_m S^*_m(B(\Ka^L_m)\otimes I_2)S_m.
\]
The unitaries $T_m:\Ka_m^L\to \Ka^L_{m\ominus_q 1}$ are given by the restrictions $T_m=Z^*|_{\Ka^L_m}$, so the action of $\Phi$ on $\Ne$ is described by the 
homogeneous cyclic shift OQRW on $\mathbb Z_q$, with local spaces $\mathbb C^2$ and unitaries $U_i\equiv U=\left(\begin{array}{cc}   1 &0 \\ 0&\omega\end{array}\right)$, $i=0,\dots,q-1$ (cf. Proposition \ref{prop:shift}).

Let us  compute the states $\rho_m$ and maps $\Xi_m$ defined in Proposition \ref{prop:dfs}.   Let $\Delta, \bar\Delta: B(\mathbb C^2)\to B(\mathbb C^2)$ be given by 
 \[
\Delta \left(\begin{array}{cc} b_{00} & b_{01} \\ b_{10} & b_{11}\end{array}\right) = \left(\begin{array}{cc} b_{00} & 0 \\ 0 & b_{11}\end{array}\right),\quad
\bar \Delta \left(\begin{array}{cc} b_{00} & b_{01} \\ b_{10} & b_{11}\end{array}\right) = \left(\begin{array}{cc} b_{11} & 0 \\ 0 & b_{00}\end{array}\right).
 \]
 It is easily checked that for each $m$, the map $\Xi_m: B(\mathbb C^2)\to B(\mathbb C^2)$ is defined as
\[
\Xi_m= \Xi :=\alpha \Delta+ (1-\alpha) \bar \Delta.
\]
It follows by Proposition \ref{prop:dfs} (i) that the states $\rho_m$  must all be equal to the unique invariant state $\rho=\frac 12 I_2$  of $\Xi$.

Let us now turn to Theorem \ref{prop:fixed_multi}. 
 We have 
  $\tilde T_{m}=Z^{-m}|_{\Ka^L_0}$, $m=0,\dots,q-1$, in particular, the unitary
 $\tilde T_{0}$ has two eigenvalues $\pm1$, with eigenvectors $|0,\pm\>$, so that 
 \[
 \tilde T_{0}=|0,+\>\<0,+|-|0,-\>\<0,-|.
\]
The subalgebra $\{\tilde T_{0}\}'\cap B(\Ka_0^L)$ of Theorem \ref{prop:fixed_multi} is the abelian subalgebra spanned by 
 the projections $\ketbra{0,\pm}{0,\pm}$. Note that we have
\[
\sum_{m=0}^{q-1} \tilde T_{m}|0,\pm\>\<0,\pm|\tilde T^*_{m} =\sum_{m=0}^{q-1} Z^{-m}|0,\pm\>\<0,\pm|Z^m=\sum_{m=0}^{q-1} |m,\pm\>\<m,\pm|=P_\pm,
\]
so that
\[
S^*T(\ketbra{0,\pm}{0,\pm}\otimes I^R)T^*S=\tilde P_\pm
\]
are the central projections of $\Fe$,
which corresponds to Theorem \ref{prop:fixed_multi} (i). For $s\in [0,1]$, put $\omega_s:=s\ketbra{0,+}{0,+}+(1-s) \ketbra{0,-}{0,-}$, then 
 we can see from  Theorem \ref{prop:fixed_multi} (ii) and (iii) that the invariant states of $\Phi$ are precisely those of the form
 \[
\xi_s:=S^*T(\omega_s\otimes \frac1d I_d)T^*S=\frac 1d (s\tilde P_++ (1-s)\tilde P_-).
 \]
Finally, let $\Psi_\pm$ be the irreducible channels on $B(\Ka^R)$ corresponding to the restrictions of $\Phi$ by the projections 
 $\tilde P_\pm$   as in Theorem \ref{prop:fixed_multi} (iv). Let $X_q$, $Z_q$ be the generalized Pauli operators on the $q$-dimensional Hilbert space with standard basis $\{|m\>\}$. One can check that we have
 \[
\Psi_+ = (\alpha \Delta+ (1-\alpha) \bar\Delta)\otimes (X_q\cdot X_q^*)= \Xi\otimes   (X_q\cdot X_q^*)
\]
and 
\[
\Psi_- = \alpha \Delta \otimes (X_q\cdot X_q^*) + (1-\alpha) \bar\Delta\otimes (X_qZ_q^*\cdot Z_qX_q^*).
\]

\subsubsection{A homogeneous nearest neighbor OQRW on ${\mathbb Z}$}


Let us consider a homogeneous nearest neighbor OQRW $\Phi$  on $V=\mathbb Z$, with local space ${\mathbb C}^2$. We will assume that the transition operators 
$L_+$, $L_-$ are  invertible and that $\Phi$ is irreducible. The last condition implies that no invariant state exists, so the results of previous sections
 cannot be applied. Nevertheless, we show that also in this case the decoherence free algebra is generated by the cyclic resolution of $\Phi$, cf. Corollary \ref{coro:W1}.

\begin{prop} Let us denote $P_{even}:=\sum_i I_2\otimes \ketbra{2i}{2i}$, $P_{odd}:=\sum_i I_2\otimes \ketbra{2i+1}{2i+1}$. Then 
$$
\Ne=\mbox{\rm span}\{ P_{odd}, P_{even}\}
$$
unless there exists an orthonormal basis $\{f_0, f_1\}$ such that $L_-$ and $L_+$ are one diagonal and one off-diagonal in this basis.
In the last case, $\Ne$ is generated by the cyclic projections 
$$
P_{\epsilon,\delta}=\sum_j (\ketbra{f_\epsilon}{f_\epsilon} \otimes \ketbra{4j+\delta}{4j+\delta}
+ \ketbra{f_{1-\epsilon}}{f_{1-\epsilon}} \otimes \ketbra{4j+2+\delta}{4j+2+\delta}),
$$
with $\epsilon,\delta=0,1$ and the period is 4. Otherwise the period is 2 with cyclic projections $P_{odd},\,  P_{even}$.
\end{prop}

The period was already computed in \cite{capa2015}.

\begin{proof}
By Corollary \ref{coro:diag}, we know that the decoherence free algebra $\Ne$ consists only of block-diagonal operators. Then a projection $P$ in $\Ne$ will have the form 
$$
P=\sum_j P_j \otimes \ketbra{j}{j},
$$
where, by Corollary \ref{coro-OQRWs-M}, $P_j$ are projections  satisfying at least the conditions
\begin{equation}\label{conditions}
P_{j} \in\{|L_+^*|, |L_-^*|\}', \qquad P_{j-1} L_- L_+^* = L_- L_+^* P_{j+1} \qquad \forall \, j.
\end{equation}
We can write the action of $\Phi$ explicitly, in particular
\begin{eqnarray}\label{Phi(P)}
\Phi(P) &=& \sum_j( L_+^*P_{j+1}L_+ + L_-^*P_{j-1}L_- ) \otimes \ketbra{j}{j},\nonumber\\
\Phi^2(P) &=&  \sum_j( L_+^{*2}P_{j+2}L_+^2 + L_-^{*2}P_{j-2}L_-^2 +L_-^*L_+^*P_jL_+L_- +L_+^*L_-^*P_jL_-L_+ ) \otimes \ketbra{j}{j}.
\end{eqnarray}
By these relations, it is easily deduced that $\Phi^n(P_{odd})$ is equal to $P_{odd}$ for even $n$ and to $P_{even}$ for odd $n$ (and similarly for $\Phi^n(P_{even})$). In particular, $\Phi^n(P_{odd}),\Phi^n(P_{even})$ are always projections and this allows us to conclude that $P_{odd}$ and $P_{even}$ belong to $\Ne$, Proposition 
\ref{prop:dfageneral} (ii). Moreover, they are trivially central, i.e., for any other projection $P$ in $\Ne$, $PP_{odd}=P_{odd}P$ and $PP_{even}=P_{even}P$.

When there exists an orthonormal basis $\{f_0, f_1\}$ such that $L_-$ and $L_+$ are one diagonal and one off-diagonal in this basis, it is easy to see that the projections $P_{\epsilon,\delta}$ in the statement are cyclic. It is a little more complicated to see that these cyclic projections can exist only in that case and anyway no other minimal projection can then appear.

So now we want to consider, for a homogeneous irreducible OQRW, whether there exists a projection $P$ in $\Ne\setminus \mbox{span}\{ P_{odd}, P_{even}\}$. We shall see that this is not possible, unless we are in the special case described in the statement.

If such a $P$ exists, then $P=PP_{odd}+PP_{even}$ and the two addends are both in $\Ne$, so, by homogeneity, it will be sufficient to search for a projection $P$ in $\Ne$ such that $P=PP_{even}$ and $0<P<P_{even}$. Then we consider $P=\sum_j P_{2j}\otimes \ketbra {2j}{2j}$.

Relations \eqref{conditions} imply that all the $P_{2j}$'s have the same rank (since the transition operators are invertible). Then, if $P$ is different from $0$ and from $P_{even}$, the only possibility is that $P_{2j}$ is a rank one projection for any $j$. Calling $u$ a norm one vector such that $P_0=\ketbra{u}{u}$, and denoting $V:=L_- L_+^*$, we deduce
$$
P=\sum_j \ketbra{V^{-j}u}{V^{*j}u} \otimes \ketbra{2j}{2j},
$$
where $V^{-j}u \parallel V^{*j}u$ because any $P_{2j}$ is a projection and, due to the first condition in \eqref{conditions}, $V^{*j}u$ is a common eigenvector of $|L_+^*|$ and $ |L_-^*|$ for any $j$.

Similar considerations will hold for $\Phi^n(P)$, but considering only odd vertices instead of even vertices when $n$ is odd. Indeed, starting with $n=1$ (for $\Phi^n(P)$ we simply proceed inductively),
\\
- $\Phi(P)$ is a projection in $\Ne$, $\Phi(P)\le P_{odd}$ due to the fact that $0\le P\le P_{even}$ and $\Phi$ is positive,
\\
- moreover , when $P\neq P_{even}$ then $\Phi(P)\neq P_{odd}$ by irreducibility; indeed, if we had for instance $P\neq P_{even}$ and $\Phi(P)=P_{odd}$, then $P_{even}-P$ would be a non-zero projection in the kernel of $\Phi$ and this contradicts irreducibility. 

Then, using \eqref{Phi(P)}, we need that
$$
\Phi^2(P) (1\otimes \ketbra{0}{0})=  ( L_+^{*2}P_{2}L_+^2 + L_-^{*2}P_{-2}L_-^2 +L_-^*L_+^*P_0L_+L_- +L_+^*L_-^*P_0L_-L_+ ) 
$$
is a one dimensional projection. This implies in particular that
$L_-^*L_+^*u\parallel L_+^*L_-^*u$, so that $u$ is an eigenvector for $(L_+^*L_-^*)^{-1}L_-^*L_+^*$.

Also, calling $u^\perp$ a norm one vector orthogonal to $u$, $P':=P_{even}-P=\sum_j \ketbra{V^{-j}u^\perp}{V^{*j}u^\perp} \otimes \ketbra{2j}{2j},$ will be a projection in $\Ne$ and so $u^\perp$ will satisfy the same conditions as $u$.

Summing up, we have that $u$ and $u^\perp$ should be two distinct eigenvectors for the operators 
\begin{equation}\label{3operators}
|L_+^*|, \qquad |L_-^*|, \qquad W:=(L_+^*L_-^*)^{-1}L_-^*L_+^*.
\end{equation}

Now, we claim that, due to irreducibility, the previous operators cannot be all proportional to the identity and we postpone of some lines the proof of this claim.

This fact implies that, either such vectors $u$ and $u^\perp$ do not exist, and so $\Ne=\mbox{\rm span}\{ P_{odd}, P_{even}\}$, or they can be chosen in a unique way, up to multiplicative constants, as the orthonormal basis which diagonalize all the three operators above. In the latter case, we now look at the form of $\Phi(P)$ given in \eqref{Phi(P)} and we see that
$$
\Phi(P) (1\otimes \ketbra{j}{j})=L_+^*P_{j+1}L_+ + L_-^*P_{j-1}L_- 
$$
should be a one dimensional projection on a vector $v$ which should be an eigenvector of the same three operators. This implies that 
$$
L_\epsilon^*u,L_\epsilon^*u^\perp \in 
{\rm span}\{ u\} \cup {\rm span}\{ u^\perp\},
\qquad \epsilon=+,-
$$
and consequently that the operators $L_+$ and $L_-$ should be either diagonal or off-diagonal in the basis 
$\{u, u^\perp\}$; but they cannot be both diagonal nor both off-diagonal, because this would contradict irreducibility. So the conclusion follows choosing $\{f_0, f_1\}=\{u, u^\perp\}$.

Finally, we go back to prove the claim. By contradiction, we suppose that all the operators in \eqref{3operators} are proportional to the identity, so that
$$
L_+=c_+U_+, \qquad L_-=c_-U_- \qquad W=c 1
$$
for some complex numbers $c_+,c_-,c$ and unitary operators $U_+,U_-$. 
Then we can rewrite
$$
W=c1= U_-U_+U_-^*U_+^*
\qquad \Rightarrow \qquad
U_-=cU_+U_-U_+^*
$$
But now write the diagonal form of the unitary $U_+$, $U_+^*=\sum_{k=0,1}\lambda_k\ketbra{v_k}{v_k}$, with $\lambda_0,\lambda_1$ in the unit circle and $\{v_0,v_1\}$ orthonormal basis, and consider
$$
\langle v_k,U_- v_j \rangle = c \langle v_k,U_+U_-U_+^* v_j \rangle
= c\overline{\lambda}_k\lambda_j \langle v_k,U_- v_j \rangle
\qquad \mbox { for } j,k=0,1.
$$
This implies $c=1$ and $\lambda_0=\lambda_1$ which requires that $U_+$ and so $L_+$ are proportional to the identity. But this contradicts irreducibility.
\end{proof}

\end{document}